\documentclass[showpacs,preprintnumbers,amsmath,amssymb,12pt]{article}

\usepackage{a4wide,amsmath,amsthm,epsfig,graphicx}
\usepackage{amsmath,amsthm,amssymb,eurosym}
\usepackage{amsfonts}
\usepackage{graphics}
\usepackage{adjustbox}
\usepackage{hyperref}
\usepackage{dsfont}
\usepackage{bbm}
\usepackage{bm}
\usepackage{xcolor}
\usepackage{dcolumn}
\usepackage{fancyhdr,fancybox}
\usepackage[svgnames]{pstricks}
\usepackage{pst-text}
\usepackage{pst-plot,pst-eucl,pstricks-add}
\usepackage{caption}
\usepackage{subcaption}

\usepackage{algorithm}
\usepackage{algpseudocode}
\newtheorem{thm}{Theorem}[section]
\newtheorem{cor}[thm]{Corollary}

\newtheorem{prop}[thm]{Proposition}

\newcommand{\sde}{\emph{SDE}}

\newcommand{\refeq}[1]{~(\ref{#1})}
\newcommand{\myref}[1]{~\ref{#1}}

\newcommand{\Z}{\mathbb{Z}}

\newcommand{\rv}{\textit{rv}}
\newcommand{\id}{\textit{id}}
\newcommand{\iid}{\textit{iid}}
\newcommand{\sd}{\textit{sd}}

\newcommand{\pdf}{\textit{pdf}}

\newcommand{\chf}{\textit{chf}}

\newcommand{\Pqo}{\bm{P}\hbox{-\emph{a.s.}}}
\newcommand{\eqd}{\stackrel{d}{=}}

\newcommand{\PR}[1]{\bm{P}\left\{{#1}\right\}}
\newcommand{\EXP}[1]{\bm{E}\left[{#1}\right]}
\newcommand{\VAR}[1]{\bm{V}\left[{#1}\right]}
\newcommand{\SK}[1]{\bm{Skew}\left[{#1}\right]}
\newcommand{\KUR}[1]{\bm{Kurt}\left[{#1}\right]}

\newcommand{\arem}{$a$-remainder}

\newcommand{\poiss}{\mathfrak{P}}
\newcommand{\pol}{\overline{\mathfrak{B}}}

\newcommand{\unif}{\mathfrak{U}}

\newcommand{\gam}{\Gamma}
\newcommand{\bgam}{\mathfrak{b}\Gamma}
\newcommand{\gampp}{\Gamma_a^{++}}
\newcommand{\bgampp}{\mathfrak{b}\Gamma_a^{++}}
\newcommand{\erl}{\mathfrak{E}}

\newcommand{\bin}{\mathfrak{B}}

\newcommand{\ou}{OU}

\newcommand{\Levy}{L\'{e}vy}

\newcommand{\gou}{$\Gamma$-OU}
\newcommand{\bgou}{$\mathrm{bi}\Gamma$-OU}

\title{\LARGE \textbf{Gamma Related Ornstein-Uhlenbeck Processes and their Simulation}
\footnote{ The views, opinions, positions or strategies
expressed in this article are those of the authors and do not
necessarily represent the views, opinions, positions or strategies
of, and should not be attributed to E.ON SE.}}

\author{Nicola \textsc{Cufaro Petroni}\footnote{cufaro@ba.infn.it}  \\
Dipartimento di \textsl{Matematica} and \textsl{TIRES}, Universit\`a di Bari\\
\textsl{INFN} Sezione di Bari\\ \vspace{7pt}
via E. Orabona 4, 70125 Bari, Italy\\
Piergiacomo
\textsc{Sabino}\footnote{piergiacomo.sabino@eon.com}\\
Quantitative Modelling \\
E.ON SE\\
\vspace{5pt}
 Br\"{u}sseler Platz 1, 45131 Essen, Germany
}

\date{}
\begin{document}
    \maketitle \thispagestyle{empty}

        \begin{abstract}
					\noindent We investigate the distributional properties of two generalized Ornstein-Uhlenbeck (OU) processes whose stationary distributions are the gamma law and the bilateral gamma law, respectively. The said
					distributions turn out to be related to the self-decomposable gamma and bilateral gamma
					laws, and their densities and characteristic functions are here given in
					closed-form. Algorithms for the exact generation of such processes
					are accordingly derived with the advantage  of being
					significantly faster than those available
					in the literature and therefore suitable for real-time simulations. 
        \end{abstract}				
        \section{Introduction and Motivation}\label{sec:intro}
In the present paper we study the distributional properties of the Gamma-Ornstein-Uhlenbeck process (\gou) 
and the Bilateral Gamma-OU process (\bgou). 
Our  contribution consists  in the derivation of the
closed-form of both the density and the characteristic function of
such processes.  In its turn, this main result enables us to
obtain fast algorithms for their exact simulation, along with
 an unbiased transition density that can be used for parameter
estimation.

To this end, following Barndorff-Nielsen and Shephard
\cite{BNSh01}, we consider a \Levy\ process $Z(t)$ and  the
generalized \ou\ process defined by the \sde
            \begin{equation}\label{eq:genOU_sde}
              dX(t) =  -kX(t)dt + dZ(t) \quad\qquad X(0)=X_0\quad \Pqo\qquad
              k>0.
            \end{equation}
 Here $Z(t)$ is called the Backward Driving \Levy\ Process
(\emph{BDLP}), and we will adopt the following notation: if
$\mathfrak{D}$ is the stationary law of $X(t)$ we will say that
$X(t)$ is a $\mathfrak{D}$-\ou\ process; if on the other hand, $Z(1)$
 (namely the \emph{BDLP} at time $t = 1$) is distributed
according to the  \id\ (infinitely divisible) law
$\widetilde{\mathfrak{D}}$, then we will say that $X(t)$ is an
\ou-$\widetilde{\mathfrak{D}}$ process. Now a well known result (see
for instance Cont and Tankov~\cite{ContTankov2004},
Sato~\cite{Sato}) is that,  a given one-dimensional
distribution $\mathfrak{D}$ always is the stationary law of a
suitable \ou-$\widetilde{\mathfrak{D}}$ process if and only if
$\mathfrak{D}$ is self-decomposable.

We recall that a law with probability density (\pdf) $f(x)$ and
characteristic function (\chf) $\varphi(u)$ is said to be
\emph{self-decomposable} (\sd) (see Sato\cite{Sato} or Cufaro
Petroni~\cite{cufaro08}) when for every $0<a<1$ we can find another
law with \pdf\ $g_a(x)$ and \chf\ $\chi_a(u)$ such that
                \begin{equation}\label{aremchf}
                    \varphi(u)=\varphi(au)\chi_a(u)
                \end{equation}
We will accordingly say that a random variable (\rv) $X$ with \pdf\
$f(x)$ and \chf\ $\varphi(u)$ is \sd\ when its law is \sd: looking
at the definition, this means that for every $0<a<1$ we can always
find two \emph{independent} \rv's, $Y$ (with the same law of $X$)
and $Z_a$ ( here called \emph{\arem}), with \pdf\ $g_a(x)$
and \chf\ $\chi_a(u)$) such that
                \begin{equation}\label{sdec-rv}
                    X\eqd aY+Z_a\qquad\quad\Pqo
                \end{equation}
								
It is well known that the \gou\ process $X(t)$ solution
of\refeq{eq:genOU_sde} implies that $Z(t)$ is a compound Poisson with
exponential jumps (see for instance Schoutens \cite{Schoutens03}) and we will prove that for the \bgou\ process $Z(t)$ is instead a compound Poisson with double exponential distribution as defined in Kou \cite{Kou2002}.

We will show  that the law of the \gou\ process and the \bgou\ process at time $t$ coincide with that
of the \arem\ $Z_a$ of a \sd\ gamma and a bilateral gamma distribution, respectively. 
Although  a
similar result has yet to be proved for other generalized \ou\
processes, in  our particular case it allows to  find the \pdf\ and
the \chf\ of $X(t)$ in closed-form because the \arem's of a gamma and a bilateral gamma
distribution turn out to be manageable mixtures of other
elementary distributions. As a consequence, we can design
efficient and fast algorithms to exactly simulate \gou\ and \bgou\ processes,
outperforming in so doing every other existing alternative (see Cont and Tankov \cite{ContTankov2004} and Qu et al. \cite{QDZ19}).
The numerical experiments we have conducted clearly show that the computational times of our approach are very small therefore, our solution is suitable for real time simulations. 

As observed in Barndorff-Nielsen and Shephard
\cite{BNSh01}, the \gou\ process is a very tractable
model that could adopted in many potential applications. For instance, in the energy and in the commodity field, many authors  (using sometimes different naming conventions) coupled a \gou\ process or a combination of \gou\ processes to a standard Gaussian-\ou\ to model day-ahead spot prices. Among others, Kluge \cite{Kluge2006} and Kjaer \cite{Kjaer2008} apply such a combination to price swing options and gas storages while Benth and Pircalabu \cite{BenthPircalabu18} apply a \gou\ process to evaluate wind derivatives. In alternative, Meyer-Brandis  and Tankov \cite{MBT2008} adopted a two regime-switching model consisting in a Gaussian-OU process and a \gou\ process to model power prices. The use of \gou\ or \bgou\ in energy market is justified by the fact that gas and power prices exhibit strong mean-reversion and spikes. Beyond commodity markets, other applications of the \gou\ and \bgou\ processes are available in the literature: among others Barndorff-Nielsen and Shephard \cite{BNSh01} used a \gou\ process to model stochastic volatility while Schoutens and Cariboni \cite{SchoutensCariboni} and Bianchi and Fabozzi \cite{Bianchi2015} adopted the \gou\ process
as a stochastic intensity process for modelling credit default risk and pricing credit default swaps.

The paper is structured as follows: in Section \ref{sec:gou} we study the distributional properties of a \gou\ process showing that it can be represented as a mixture of Polya or binomial mixtures and therefore, it can be seen as a compound sum of independent exponential \rv's or as an Erlang \rv\ with a random index. These findings are instrumental to design the simulation algorithms illustrated in Subsection \ref{subsec:simulation:gou}. 
Section \ref{sec:bgou} analyzes the distributional properties of a \bgou\ process and focuses on the case with symmetric parameters. These results are then used in Subsection \ref{subsec:simulation:bgou} to concieve the relative simulation algorithms. Section \ref{sec:sim:exper} illustrates the numerical experiments that we have conducted to compare the convergence and computational performance of our solutions to the approaches in Cont and Tankov \cite{ContTankov2004} and in Qu at al. \cite{QDZ19}. Finally Section\myref{sec:conclusions}
concludes the paper with an overview of future inquiries and
possible further applications.

\section{Distributional properties of the \gou\ process\label{sec:gou}}

According to an aforementioned result, a \gou($k, \lambda,
\beta$) process $X(t)$ is the solution of\refeq{eq:genOU_sde} where the \emph{BDLP} is a
compound Poisson process $Z(t)$  with intensity $\lambda$ of
the number process $N(t)$, and identically distributed exponential
jumps $J_n\sim\erl_1(\beta)$
            \begin{equation*}
              Z(t)=\sum_{n=0}^{N(t)}J_n\qquad\quad
              J_0=0\quad\Pqo
            \end{equation*}
It turns out that $Z(t)$ is a
subordinator, and that the solution of\refeq{eq:genOU_sde} reads
now as
            \begin{equation}\label{eq:sol:Comp:OU}
              X(t)= x_0e^{-kt} + \sum_{n=0}^{N(t)}e^{-k(t-\tau_n)}J_n
            \end{equation}
where $\tau_n\;(\tau_0=0,\,\Pqo)$ are the jump times of the Poisson
process $N(t)$. Following Cont and Tankov \cite{ContTankov2004} and
Kluge\cite{Kluge2006} it results that the \chf\ of $X(t)$,  with
$X(0)=0\;\Pqo$, is
            \begin{equation}\label{eq:chf:Comp:OU}
              \varphi(u,t) = \left(\frac{\beta - iue^{-kt}}{\beta - iu}\right)^{\frac{\lambda}{k}}
            \end{equation}
As it will be discussed in the next section, this coincides with
the \chf\ of the $e^{-kt}$-remainder of the gamma law
$\gam\big(\,\!^{\lambda}/_k\,,\,\beta\big)$  which is famously
\sd, while its stationary distribution
            \begin{equation*}
              \varphi_s(u) = \left(\frac{\beta }{\beta - iu}\right)^{\frac{\lambda}{k}}
            \end{equation*}
is instead recovered for $t\rightarrow +\infty$ and it coincides
with the \chf\ of the previous gamma law (see also Barndorff-Nielsen
and Shephard\cite{BNSh01}, Grigelionis\cite{Gri03}).
The above result can be also summarized by the following theorem whose proof is a straightforward application of the homogeneity of the Poisson process.
\begin{thm}\label{th:chf:gou}
The \chf\ of $X(t+s)$ conditional on $X(s)$ is given by 
\begin{equation}
\EXP{e^{iuX(t+s)}|X(s)}=e^{iuX(s)e^{-kt}}\times \left(\frac{\beta - iue^{-kt}}{\beta - iu}\right)^{\frac{\lambda}{k}} 
\label{eq:chfgou}
\end{equation}
\end{thm}
\noindent An alternative version of Theorem \ref{th:chf:gou} can be found in Qu et al \cite{QDZ19}
\begin{thm}\label{th:chf:gou:qdz}
\begin{equation}
\EXP{e^{iuX(t+s)}|X(s)}=e^{iuX(s)e^{-kt}}\times e^{\lambda t \left(\varphi_{\tilde{J}}(u)-1\right)}
\label{eq:chf:qdz}
\end{equation}
where
\begin{equation}
\varphi_{\tilde{J}}(u)=\int_0^1\frac{\beta e^{ktv}}{\beta e^{ktv} - iu}dv,
\label{eq:chf:exp:qdz}
\end{equation}
\end{thm}
\noindent hence the right-hand side in\refeq{eq:chf:qdz} is the \chf\ of compound Poisson whose jumps are independent copies of the \rv's $\tilde{J}_i$ distributed according to a uniform mixture of exponential laws with random parameter $\beta e^{ktU}$ and $U\sim\unif([0,1])$. 
\subsection{Polya mixtures of gamma laws $\gam(\alpha, \beta)$\label{subsec:gammamix}}
Due to the fact that the stationary law of a \gou\ is a gamma law, it is natural to investigate how it is related to the law of the process at time $t$.
We recall that the laws of the two gamma family $\gam(\alpha,\beta)$
($\alpha>0,\beta>0$) have the following \pdf\ and \chf
\begin{eqnarray}
  f_{\alpha,\beta}(x) &=& \frac{\beta}{\Gamma(\alpha)}(\beta x)^{\alpha-1}e^{-\beta x} \qquad\qquad x>0\label{gammapdf}\\
  \varphi_{\alpha,\beta}(u)  &=&
  \left(\frac{\beta}{\beta-iu}\right)^\alpha\label{gammachf}
\end{eqnarray}

\noindent In particular $\gam(k,\beta)$, with
$\alpha=k=1,2,\ldots$ a natural number, are the Erlang laws
$\erl_k(\beta)$, and $\gam(1,\beta)$  is the usual exponential
law $\erl_1(\beta)$.  The $\gam(\alpha,\beta)$ laws are \sd\
(see Grigelionis\cite{Gri03}), so that from\refeq{aremchf} the law
of their \arem\ $Z_a$ has the \chf
            \begin{equation}\label{eq:chfGammaAdjoint}
                \chi_a(u; \alpha, \beta)=\frac{\varphi_{\alpha, \beta}(u)}{\varphi_{\alpha, \beta}(au)} = \left(\frac{\beta-iau}{\beta-iu}\right)^{\alpha}
            \end{equation}
 It is apparent now from\refeq{eq:chfGammaAdjoint} and Theorem \ref{th:chf:gou} that the
\chf\refeq{eq:chf:Comp:OU} of a \gou($k, \lambda, \beta$) process at time $t$ 
is that of the \arem\ (hereafter dubbed $\gampp(\alpha,\beta)$ law) of a  $\gam(\alpha,\beta)$ law plus a constant $a x_0$ when we take $a =
e^{-k t}$ and $\alpha=\lambda/k$. 

The moments of $Z_a$, as well as those of a \gou\ process, can be obtained simply deriving the \chf\, however, it is easier to work with the cumulants of $Z_a$ that can be calculated with a straightforward application of the properties of the cumulant generating function (the logarithm of the moment generating function)
\begin{equation}
\kappa_n(Z_a) = (1-a^n) \kappa_n(X)
\label{eq:cum:arem}
\end{equation} 
where $\kappa_n(Z_a)$ and $\kappa_n(X)$ represent the $n$-th cumulant of the \arem\ and a gamma distributed \rv\ $X$, respectively. We remark that\refeq{eq:cum:arem} is applicable to the cumulants of the \arem\ of any \sd\ distribution.
After some algebra, it results that the expected value, the variance, the skewness and the kurtosis of $Z_a$ are 
\begin{eqnarray}
\EXP{Z_a} &=& (1-a)\frac{\alpha}{\beta}\\
\VAR{Z_a} &=& (1-a^2) \frac{\alpha}{\beta^2}\\
\SK{Z_a} &=&\frac{1-a^3}{(1-a^2)^{3/2}}\times\frac{2}{\sqrt{\alpha}}\\
\KUR{Z_a} &=&\frac{1+a^2}{1-a^2}\times\frac{6}{\alpha}+3.
\label{eq:mom:gou}
\end{eqnarray} 
Of course $\EXP{X(t)}=ax_0 + \EXP{Z_a}$ while the variance, the skewness and kurtosis of $X(t)$ and $Z_a$ coincide because these quantities are translation invariant. 
It is interesting to note that the laws $\gam(\alpha, \beta)$ and $\gampp(\alpha, \beta)$ share the same summation and scaling properties.
\begin{prop}\label{prop:scale:sum}
\leavevmode
\makeatletter
\@nobreaktrue
\makeatother
	\begin{enumerate}
		\item If $Z_a\sim\gampp(\alpha, \beta)$ then for any $c>0$, 
		\begin{equation*}
			cX\sim\gampp\left(\alpha, \frac{\beta}{c}\right).
		\end{equation*}
		\item If $Z_{a,i}\sim\gampp(\alpha_i, \beta), \,i=1, \dots, N$ and independent then
		\begin{equation*}
			\sum_{i=1}^NZ_{a,i}\sim\gampp\left(\sum_{i=1}^N\alpha_i, \beta\right).
		\end{equation*}
	\end{enumerate}
\end{prop}
\begin{proof}
\item The \chf\ of $cZ_a$ is 
	\begin{equation*}
		\chi_a(cu) = \left(\frac{\beta-iacu}{\beta-icu}\right)^{\alpha}=\left(\frac{\frac{\beta}{c}-iau}{\frac{\beta}{c}-iu}\right)^{\alpha},	
	\end{equation*}
	that is the \chf\ a $\gampp(\alpha, \frac{\beta}{c})$ distributed \rv.
\item The \chf\ $\bar{\chi_a}(u)$ of $\sum_{i=1}^NZ_{a,i}$ is
	\begin{equation*}
		\bar{\chi_a}(u) = \EXP{e^{iu\sum_{i=1}^NZ_{a,i}}} = \prod_{i=1}^N \left(\frac{\beta-iau}{\beta-iu}\right)^{\alpha_i}=	\left(\frac{\beta-iau}{\beta-iu}\right)^{\sum_{i=1}^N\alpha_i},
	\end{equation*}
	that coincides with the \chf\ of a $\gampp(\sum_{i=1}^N\alpha_i,\beta)$ law and that concludes the proof.
\end{proof}

In order to further investigate the distributional properties of the law of the \arem\ and of the law of a \gou\ process, we now consider a \rv\ $S$ distributed according to a
\emph{negative binomial, or Polya distribution}, denoted hereafter
$\pol(\alpha,p)\,,\; \alpha>0,\,0<p<1$,  namely such
that
            \begin{equation*}
                \PR{S=k} =  \binom{\alpha+k-1}{k}(1-p)^\alpha p^k\qquad\quad k=0,1,\ldots
            \end{equation*}
Remark that,  when $\alpha=n=1,2,\ldots$ is a natural number,
the Polya distribution $\pol(n,p)$ coincides with the so called
\emph{Pascal distribution}, and in particular $\pol(1,p)$ is nothing
 else than the usual \emph{geometric distribution} $(1-p)p^k$.
 From the generalized binomial formula it is possible to see
now that its \chf\ is
            \begin{equation*}\label{polya}
              \varphi_S(u) =\sum_{k=0}^\infty\binom{\alpha+k-1}{k}(1-p)^\alpha p^ke^{iuk}= \left(\frac{1-p}{1-p\,e^{iu}}\right)^\alpha.
            \end{equation*}
where the series -- that certainly converges because
$|p\,e^{iu}|=p<1$ -- has the form of an infinite Polya
$\pol(\alpha,p)$-weighted mixture of degenerate laws.

As observed  for instance in Panjer and Wilmott
\cite{panjer_willmot_1981}, this result can also be extended
by taking the \rv's
            \begin{equation*}
                Z=\sum_{j=0}^SX_j
            \end{equation*}
sums of a random number  $S\sim \pol(\alpha,p)$ of \iid\ \rv's $X_j$
with the common \chf\ $\varphi_X(u)$, and $X_0=0,\;\Pqo$:  in
this case we have indeed
            \begin{eqnarray}\label{Z}
              \varphi_Z(u) &=& \EXP{e^{iuZ}}=\EXP{\EXP{\left.e^{iuZ}\right|S}}\nonumber \\
               &=&\sum_{k=0}^\infty\binom{\alpha+k-1}{k}(1-p)^\alpha
               p^k\,\EXP{e^{iu\sum_{j=0}^kX_j}} \\
               &=&(1-p)^\alpha\sum_{k=0}^\infty\binom{\alpha+k-1}{k}
               p^k\varphi_X(u)^k=\left(\frac{1-p}{1-p\,\varphi_X(u)}\right)^\alpha\nonumber
            \end{eqnarray}
where again  the series converges because
$|p\,\varphi_X(u)|\le p<1$. This shows that the law of $Z$ is
again an infinite Polya $\pol(\alpha,p)$-weighted mixture of laws
$\varphi_X(u)^k$: if these laws also have a known \pdf, then the law
of $Z$ too has an explicit representation as a mixture of
\pdf's

\begin{thm}\label{th:arem:density}
The law of the \arem\ of the $\gam(\alpha,\beta)$ law is an infinite Polya
$\pol(\alpha,1-a)$-weighted mixture of Erlang laws
$\erl_k\big(\,\!^\beta/_a\big)$ with the following \chf\ $\chi_a(u, \alpha, \beta)$ and density $g_a(x, \alpha, \beta)$
\begin{equation}
\chi_a(u, \alpha, \beta)=\sum_{k=0}^\infty\binom{\alpha+k-1}{k}a^\alpha(1-a)^k\left(\frac{\beta}{\beta-iau}\right)^k
\label{eq:polya:gamma:chf}
\end{equation}
\begin{equation}\label{eq:mix:polya}
	g_a(x, \alpha, \beta)=a^\alpha\delta(x)+\sum_{k=1}^\infty\binom{\alpha+k-1}{k}a^\alpha(1-a)^kf_{k,\,\!^\beta/_a}(x) \qquad\qquad x>0
\end{equation}						
\end{thm}
\begin{proof}
By taking now  $p=1-a$ and
$X\sim\erl_1\big(\,\!^\beta/_a\big)$ an exponential with \chf\ 
\begin{equation*}
    \varphi_X(u)=\frac{\beta}{\beta-iau}
\end{equation*}

\noindent it is easy to see from\refeq{eq:chfGammaAdjoint}
and\refeq{Z} that
            \begin{eqnarray*}
              \left(\frac{\beta-iau}{\beta-iu}\right)^\alpha &=& \left(\frac{a(\beta-iau)}{\beta-iau-(1-a)\beta}\right)^\alpha
                    =\left(\frac{a}{1-(1-a)\frac{\beta}{\beta-iau}}\right)^\alpha \\
               &=&\sum_{k=0}^\infty\binom{\alpha+k-1}{k}a^\alpha(1-a)^k\left(\frac{\beta}{\beta-iau}\right)^k
            \end{eqnarray*}
that is the \chf\ of an infinite Polya
$\pol(\alpha,1-a)$-weighted mixture of Erlang laws
$\erl_k\big(\,\!^\beta/_a\big)$.
This distribution can also be
considered either as an Erlang law $\erl_S\big(\,\!^\beta/_a\big)$
with a Polya $\pol(\alpha,1-a)$-distributed random index $S$,  or
even as that of a sum of a Polya random number of \iid\ exponentials
            \begin{equation*}
                \sum_{j=0}^SX_j\qquad\qquad
                S\sim\pol(\alpha,1-a)\qquad
                X_j\sim\erl_1\big(\,\!^\beta/_a\big)\qquad X_0=0,\;\Pqo
            \end{equation*}
 Since on the other hand from\refeq{gammapdf} the \pdf's of the
Erlang laws $\erl_k\big(\,\!^\beta/_a\big)$ are known, also the
\pdf\ of the \arem\ $Z_a$ of a gamma law $\gam(\alpha,\beta)$ is
the following explicit mixture plus a degenerate in $x=0$
\begin{equation*}
		g_a(x, \alpha, \beta)=a^\alpha\delta(x)+\sum_{k=1}^\infty\binom{\alpha+k-1}{k}a^\alpha(1-a)^kf_{k,\,\!^\beta/_a}(x) \qquad\qquad x>0
\end{equation*}
that concludes the proof.
\end{proof}
\noindent The above results give a closed-form representation of the transition density of a \gou\ process. 
\begin{cor}
The transition density $p(x, t+s| y, s) $of the \gou$(k, \lambda, \beta)$ is
\begin{equation}
p(x, t+s| y, s) = g_a\left(x-ay, \frac{\lambda}{k}, \beta\right), \quad a=e^{-kt}. 
\label{eq:gou:tdensity}
\end{equation}
where $g_a(\cdot,\lambda/k, \beta)$ is defined in\refeq{eq:mix:polya}.
\end{cor} 			
Although the parameters estimation is not the focus of our study, knowing the transition density in closed-form gives a remarkable advantage compared to the results in Qu et al. \cite{QDZ19} because one can write the log-likelihood and maximize it explicitly. 
Of course, in any practical applications,  some series truncation
rule must be  adopted but it can however be easily fine tuned.

To this end, in discrete time, a \gou\ process is equivalent to a GAR(1) auto-regressive process introduced by Gaver and Lewis \cite{gaver_lewis_1980} 
whose parameter estimation based on the EM algorithm has been discussed in Popovici and M. Dumitrescu \cite{PD2010} (for $\lambda/k$ integer only, see next section).
In alternative, one could adopt the generalized method of moments using Equations\refeq{eq:mom:gou} and obtain the associated Yule-Walker equations.  

\subsection{Binomial mixtures\label{subsec:erlangRV}} 
It follows from the previous subsection that
for $\alpha=n=1,2,\ldots$ the \arem\ of the Erlang laws
$\gam(n,\beta)=\erl_n(\beta)$ is an infinite mixture of Erlang
$\erl_k\big(\,\!^\beta/_a\big)$ with Pascal weights $\pol(n,1-a)$,
while for $n=1$ the \arem\ of the exponential law
$\gam(1,\beta)=\erl_1(\beta)$ is an infinite mixture of Erlang
$\erl_k(\beta)$ with geometric weights $\pol(1,1-a)$. In these
 two cases, however, it is  easy to see that there is an
alternative decomposition of the \arem\ law into a \emph{finite},
binomial mixture of Erlang laws.
\begin{thm}\label{th:arem:bin:density}
The law of the \arem\ of the $\erl_n(\beta)$ law is a finite mixture of Erlang $\erl_k(\beta)$ with binomial
weights $\bin(n,1-a)$ with the following \chf\ $\chi_a(u, \alpha, \beta)$ and density $g_a(x, \alpha, \beta)$
\begin{equation}
\chi_a(u, n, \beta)=\sum_{k=0}^n\binom{n}{k}a^{n-k}(1-a)^k\left(\frac{\beta}{\beta-iu}\right)^k
\label{eq:bin:erl:chf}
\end{equation}
\begin{equation}\label{eq:mix:bin}
	g_a(x, n, \beta)=a^\alpha\delta(x)+\sum_{k=1}^n\binom{n}{k}a^{n-k}(1-a)^kf_{k,\beta}(x) \qquad\qquad x>0
\end{equation}						
\end{thm}
\begin{proof}
$\alpha=n$ we have indeed
from\refeq{eq:chfGammaAdjoint}
                \begin{equation}\label{binmix}
                  \left(\frac{\beta-iau}{\beta-iu}\right)^n =
                  \left(a+(1-a)\frac{\beta}{\beta-iu}\right)^n
                        =\sum_{k=0}^n\binom{n}{k}a^{n-k}(1-a)^k\left(\frac{\beta}{\beta-iu}\right)^k
                \end{equation}
namely a finite mixture of Erlang $\erl_k(\beta)$ with binomial
weights $\bin(n,1-a)$, or in other words an Erlang law $\erl_S(\beta)$ with a binomial $\bin(n,1-a)$-distributed random index
$ S$, that is a sum
                \begin{equation*}
                    Z_a\eqd\sum_{j=0}^S X_j\qquad\qquad {S}\sim\bin(n,1-a)
                \end{equation*}
                of $S$ \iid\ exponentials $X_j\sim\erl_1(\beta)$ with
                $X_0=0,\;\Pqo$
This ambiguity in the mixture representation of a law is apparently
allowed because in general  a mixture decomposition is not
unique.
Once again, the \pdf's of the
Erlang laws are known therefore the density is simply given by\refeq{eq:mix:bin} that concludes the proof.
\end{proof}
\noindent The above results lead to a closed-form representation of the transition density of a \gou\ process (or better Erlang-\ou process) when $\lambda/k=n$. 
in terms of a finite sum of Erlang densities plus a degenerate term.
\begin{cor}
The transition density $p(x, t+s| y, s) $of the \gou$(k, \lambda, \beta)$ with $\lambda/k=n$ and $n\in\mathbb{N}^*$ is
\begin{equation}
p(x, t+s| y, s) = g_a(x-ay, n, \beta), \quad a=e^{-kt}. 
\label{eq:erlou:tdensity}
\end{equation}
where $g_a(\cdot, n, \beta)$ is defined in\refeq{eq:mix:bin}.
\end{cor} 					

 The said binomial decomposition, however, while legitimate
for $\alpha=n$, cannot be extended to the general case of
$\alpha>0$. While indeed -- always from the generalized
binomial formula -- the  following infinite decomposition of
$\chi_a(u, \alpha, \beta)$ in\refeq{eq:chfGammaAdjoint}
                \begin{eqnarray}\label{pseudomix}
  \left(\frac{\beta-iau}{\beta-iu}\right)^\alpha\!\! &=&
  \left(a+(1-a)\frac{\beta}{\beta-iu}\right)^\alpha
        =a^\alpha\left(1+\frac{1-a}{a}\frac{\beta}{\beta-iu}\right)^\alpha\nonumber\\
  &=&a^\alpha\sum_{k=0}^\infty
  \binom{\alpha}{k}\left(\frac{1-a}{a}\frac{\beta}{\beta-iu}\right)^k=
  \sum_{k=0}^n\omega_k(a,\alpha)\left(\frac{\beta}{\beta-iu}\right)^k
  \\
  \omega_k(a,\alpha)&=&\binom{\alpha}{k}a^{\alpha-k}(1-a)^k\nonumber
                \end{eqnarray}
looks again as another infinite mixture of Erlang laws
$\erl_k(\beta)$,  we must remark that first this expansion
definitely converges exclusively when it is
                \begin{equation*}
                    \left|\frac{1-a}{a}\frac{\beta}{\beta-iu}\right|\le\frac{1-a}{a}<1
                \end{equation*}
which, for $0<a<1$, only happens if $\frac{1}{2}\le a<1$;  and
second, and mainly, that although the infinite sequence of the
$\omega_k(a,\alpha)$ sums up to one, the generalized binomial
coefficients
\begin{equation*}
    \binom{\alpha}{k}=\frac{\alpha(\alpha-1)\ldots(\alpha-k+1)}{k!}\qquad\qquad
     \binom{\alpha}{0}=1
\end{equation*}
take also negative values for $k>\alpha+1$, and hence the
$\omega_k(a,\alpha)$ not always constitute a legitimate probability
distribution. As a consequence,  the decomposition\refeq{pseudomix}
is not in general a  true mixture, even if it holds
mathematically whenever it converges. In other words (as an
alternative to\refeq{eq:mix:polya}) the \pdf\ of the \arem\ $Z_a$
can always be represented also as the following combination --
let us call it a \emph{pseudo-mixture} -- of Erlang \pdf's
                \begin{equation}\label{eq:densityGammaAdjoint}
                  g_a(x) = a^\alpha\delta(x) + \sum_{k\ge 1}^{\infty}\omega_k(a, \alpha) f_{k, \beta}(x), \qquad \frac{1}{2}\le
                  a<1
                \end{equation}
 that can be interpreted as a  true mixture only when
$\alpha$ is an integer and the sum  is cut down to a finite
number of terms.

\subsection{Simulation Algorithms\label{subsec:simulation:gou}}
The results of the previous sections show that the
\chf\refeq{eq:chf:Comp:OU} of an \gou$(k,\lambda,\beta)$
coincides with that of the \emph{a}-remainder $Z_a$ of a gamma
law $\gam(\alpha,\beta)$ by simply taking $a=e^{-k\Delta t}$ and $\alpha=\lambda/k$. Algorithm \ref{alg:sd} summarizes then the procedure to generate the skeleton of a
\gou$(k,\lambda,\beta)$ process over a time grid $t_0, t_1,\dots, t_M$, $\Delta t_m = t_m - t_{m-1}\,,\; m=1,\dots,M $.
\begin{algorithm}
\caption{ }\label{alg:sd}
		\begin{algorithmic}[1]
		\For{ $m=1, \dots, M$}
		\State $\alpha\gets\lambda/k,\;\;  a\gets e^{-k\Delta t_m}$
		\State $b\gets B\sim \pol(\alpha,1-a)$ \Comment{Generate a Polya $(\alpha,1-a)$ \rv}
		\State $z_a^m \gets Z_a^{(m)}\sim\erl_{b}\left(\,\!^\beta/_a\right)$; \Comment{ Generate an Erlang \rv\ with rate $\beta/_a$}
		\State $X(t_m)\gets a\,X(t_{m-1}) + z_a^{(m)}$.
		\EndFor
		\end{algorithmic}
\end{algorithm}

The simulation of $Z_a$ is very simple and it is
applicable with no parameter  constraints.  It is worthwhile noticing that such an algorithm resembles to the one proposed in McKenzie \cite{McK87} however having the advantage to simulate exponential \rv's. When in particular
$\lambda/k = \alpha$ is an integer $n$, the steps three and four in Algorithm \ref{alg:sd} can be replaced with those in Algorithm \ref{alg:bin}
 \begin{algorithm}
 \caption{ }\label{alg:bin}
		\begin{algorithmic}[1]
		\setcounter{ALG@line}{2}
		\State $b\gets B\sim \bin(n,1-a)$ \Comment{Generate a Binomial \rv}
		\State $z_a^m \gets Z_a^{(m)}\sim\erl_{b}\left(\beta\right)$; \Comment{ Generate an Erlang \rv\ with rate $\beta$}
		\end{algorithmic}
\end{algorithm}

Of course the  assumption $\lambda/k = n$ becomes acceptable
for a fairly large $n$, namely for an \gou\ with  either a
low mean-reversion rate  or a high number of jumps.  In
other words this approximation could be used if $\lambda\gg k$,
or better when the integer part $\lfloor \lambda/k\rfloor$ is much
larger than its remainder.  On the other hand, such a
conjecture is justified by the fact that in practice every
estimation procedure presents estimation errors.

The simulation of  the $Z_a^{(m)}$, could also be
implemented starting from  the
representation\refeq{eq:densityGammaAdjoint} of their density. Over
the usual time grid  the constraint $\frac{1}{2}\le a < 1$
implies that  $k<\log2/\Delta t_m$. For instance, in energy markets and financial
applications it is common to assume $\Delta t_m=1/365$  or $\Delta
t_m<1=252$ that correspond to $k<253$ or $k<175$ respectively,
values that virtually cover all the realistic market
conditions.

Under this parameter constraint we can conceive an
acceptance-rejection procedure based on the method  of Bignami
and de Matteis\cite{BdM1971} for  pseudo-mixtures with non
positive terms (see also Devroye\cite{Dev86} page 74). Denoting
 indeed $\omega_k(a, \alpha)^+ = \max\{\omega_k(a, \alpha),
0\}$ and $\omega_k(a, \alpha)^- = \min\{\omega_k(a, \alpha), 0\}$,
 so that $\omega_k(a, \alpha) = \omega_k(a, \alpha)^+ +
\omega_k(a, \alpha)^-$, the approach of Bignami and de Matteis
relies on the  remark that from\refeq{eq:densityGammaAdjoint}
we have
\begin{equation}\label{eq:accRej}
	g_a(x) \le \sum_{k\ge 0}^{\infty}\omega_k(a, \alpha)^+ f_{k, \beta}(x) = \overline{g}(x) = c g(x)
\end{equation}
where
\begin{equation}\label{mix}
   1<c =\sum_{k\ge 0}^{\infty}\omega_k(a, \alpha)^+<\infty\qquad p_k =\frac{\omega_k(a, \alpha)^+}{c} \qquad g(x)=\sum_{k\ge 0}^{\infty}p_k f_{k,
   \beta}(x)
\end{equation}
 so that $g(x)$ turns out to be a true mixture of Erlang laws,
namely the \pdf\ of
\begin{equation*}
    V=\sum_{i=0}^{S}X_i\sim\erl_S(\beta)\quad\qquad X_i\sim\erl_1(\beta)\quad\qquad\PR{S = k} = p_k,
\end{equation*}
The generation of $\Z_a$ in the steps three and four in Algorithm \ref{alg:sd:gamma} can then be implemented employing the following acceptance-rejection solution.
\begin{algorithm}
\caption{$\frac{1}{2}\le (a <1$) }\label{alg:sd:gamma}
		\begin{algorithmic}[1]
		\State Generate $S$ with law $\PR{S = k} = p_k, k=0,\dots, N$
		\While{$u\le \frac{g_a(\bar{z})}{\bar{g}(\bar{z})}$}
		\State $u\gets U\sim\mathcal{U}[0, 1]$ \Comment{Generate a uniform}
		\State $\bar{z}\gets\bar{Z}\sim\erl(S,1)$\Comment{Generate a standard Erlang}
		\EndWhile
		
		\Return $\beta z$
		\end{algorithmic}
\end{algorithm}

The computational performance of  this algorithm can be
assessed by observing that for relatively small  values of
$\alpha$ the probability $\PR{S}=0$ is high, hence $V$ and
$Z_a$ turn out to be degenerate,  so that $Z_a$ can be
set to $0$ as well because the acceptance condition is always
satisfied.  Since on the other hand the efficiency of the
acceptance-rejection algorithm depends of the constant $c$
in\refeq{mix}, and $1/c$ roughly represents the probability of
accepting $\erl_S(1)$, it is  also preferable to have $c$ as
close  to $1$ as possible.

 Remark that for $0<\alpha\le1$ and $^1/_2\le a<1$  we always
have $\omega_0(a, \alpha)^+=a^{\alpha}\ge0.5$ with the minimum value
$0.5$ attained for $a=0.5,\;\alpha=1$, which coincides  with
the simulation of $Z\sim\erl_{\bin(1,1-a)}(1)$ (see Cufaro
Petroni and Sabino\cite{cs17}).  This means that the
concentration of the weights $\omega_k(a, \alpha)$ is mainly
around $\omega_0(a, \alpha)$  (which is a positive number) because
in the said range of $a,\alpha$ the negative coefficients
$\omega_k(a, \alpha)^-$ are rather negligible; for instance, setting
$N=40$, we find
\begin{align*}
    &\hbox{for $\alpha=0.1$ we find}\;\left\{
                                      \begin{array}{ll}
                                        c\simeq1.1311,\;\,\!^1/_c\simeq0.8841 & \;\hbox{when $\;a=0.5$} \\
                                        c\simeq1.0006,\;\,\!^1/_c\simeq0.9995 & \;\hbox{when $\;a=0.9$}
                                      \end{array}
                                    \right.
     \\
    &\hbox{for $\alpha=0.9$ we find}\;\left\{
                                      \begin{array}{ll}
                                        c\simeq1.0348,\;\,\!^1/_c\simeq0.9663 & \;\hbox{when $\;a=0.5$} \\
                                        c\simeq1.0005,\;\,\!^1/_c\simeq0.9994 & \;\hbox{when $\;a=0.9$}
                                      \end{array}
                                    \right.
\end{align*}

It is  apparent then that for $0<\alpha<1$, the
acceptance-rejection  method is very efficient because the law
of $Z_a$ is \emph{similar} to  that of $V$. If on
the other hand $\alpha>1$,  taking $n=\lfloor \alpha\rfloor$
and $\gamma$ its remainder, $Z_a$ can be also seen (and
generated) as the sum of $Z_1 + Z_2$ with
$Z_1\sim\erl_{\bin(n, 1-a)}(1)$ and $Z_2$ with \chf\ in
equation\refeq{eq:chfGammaAdjoint} with $\alpha=\gamma$. In any case our numerical
experiments  will show that $c$ is very close to $1$
also for $\alpha>1$.

We benchmark the performance of our algorithms to two alternatives available in the literature. For instance, the exact sequential
simulation of a \gou process can be achieved using the simulation procedure introduced in Lawrence \cite{L82} that coincides with the modifying Algorithm
6.2 page 174 in Cont and Tankov\cite{ContTankov2004} as detailed in Algorithm \ref{alg1:ref}
\begin{algorithm}
\caption{ }\label{alg1:ref}
\begin{algorithmic}[1]
		\For{ $m=1, \dots, M$}
		\State Generate $n\sim\poiss(\lambda\Delta t_m)$,\Comment {Poisson \rv~with intensity $\lambda\Delta t_m$}
		\State Generate $n$ \iid\ uniform \rv's $\bm{u}=(u_1, \dots, u_n)\,\sim\,\unif([0,1]^n)$.
		\State Sort $\bm{u}$, $(u_{[1]}< \dots< u_{[n]})$,
		\State $\tau_i\gets\Delta t_m u_{[i]},\, i=1,\dots n$,
		\State Generate $n$ \iid\ $J_n\sim\erl_1(\beta)$, \Comment {Exponential \rv~ with scale $\beta$}
		\State $X(t_m)\gets X(t_{m-1})e^{-k\Delta t_m} + \sum_{i=1}^{n}e^{-k(\Delta t_m-\tau_i)}J_i$.
		\EndFor
		\end{algorithmic}
\end{algorithm}
						
This solution does not directly rely on the
statistical properties described by the \chf\refeq{eq:chf:Comp:OU},
but it is rather based on the definition of the process\refeq{eq:sol:Comp:OU}. In contrast to Algorithm \ref{alg1:ref}, our approach has the obvious advantage of not requiring to draw the complete skeletons of the jump times between two time steps.

The second alternative, summarized in Algorithm \ref{alg:qdz}, is the exact simulation approach recently illustrated in Qu et al. \cite{QDZ19} that is based on Theorem \ref{th:chf:gou:qdz}.
\begin{algorithm}
\caption{ }\label{alg:qdz}
\begin{algorithmic}[1]
		\For{ $m=1, \dots, M$}
		\State Generate $n\sim\poiss(\lambda\Delta t_m)$,\Comment {Poisson \rv~with intensity $\lambda\Delta t_m$}
		\State Generate $n$ \iid\ uniform \rv's $\bm{u}=(u_1, \dots, u_n)\,\sim\,\unif([0,1]^n)$.
		\State $\beta_i\gets\beta e^{k\Delta t_m u_i}, i=1,\dots, n$.
		\State Generate $n$ \iid\ $\tilde{J}_i\sim\erl_1(\beta_i), i=1,\dots, n$, \Comment {Exponential \rv's with random rate $\beta_i$}
		\State $X(t_m)\gets X(t_{m-1})e^{-k\Delta t_m} + \sum_{i=1}^{n}\tilde{J}_i$.
		\EndFor
		\end{algorithmic}
\end{algorithm}

Algorithm \ref{alg:qdz} avoids simulating the jump times of the Poisson process as well but still requires additional steps compared to Algorithm \ref{alg:sd} which, as we will show in Section \ref{sec:sim:exper}, is by far the best performing alternative.

In addition, it is also worthwhile noticing that in the literature
several simulation algorithms based on the knowledge of the \chf\
are available (see for instance Devroye\cite{Dev86} pag 695,
Devroye\cite{Dev81} and Barabesi and Pratelli
\cite{BarabesiPratelli14}). Unfortunately, all these
algorithms require some regularity conditions on the \chf\
(absolutely integrability, absolutely continuity and absolutely
integrability of first two derivatives), that are not fulfilled by
the \chf\refeq{eq:chf:Comp:OU}.


\section{Distributional properties of the \bgou\ process\label{sec:bgou}}
The $\bgam$ distribution with parameters $\alpha_1, \beta_1, \alpha_2, \beta_2$ has been explored by K\"{u}chler and  Tappe \cite{KT2008} in the context of financial mathematics. Such a distribution can be seen as the law of the difference $X^{(u)}-X^{(d)}$ of two independent \rv's $X^{(u)}$, $X^{(d)}$ with $X^{(u)}\sim\gam(\alpha_1,\beta_1)$ and $X^{(u)}\sim\gam(\alpha_2,\beta_2)$. Therefore the \chf\ of the $\bgam$ distribution is
\begin{equation}
\varphi(v) = \left(\frac{\beta_1}{\beta_1-iv}\right)^{\alpha_1}\left(\frac{\beta_2}{\beta_2+iv}\right)^{\alpha_2} =\left(\varphi_u(v)\right)^{\alpha_1}\left(\varphi_d(v)\right)^{\alpha_2}
\label{eq:}
\end{equation}
K\"{u}chler and S. Tappe \cite{KT2008} have shown that such distribution is \sd\ therefore it is a suitable stationary law of a generalized \ou\ process. Based of the definition of \sd\ distributions, the \chf\ of the \arem\ of the $\bgam$\ law is
\begin{equation}
\chi_a(v) = \left(\frac{\beta_1-iav}{\beta_1-iv}\right)^{\alpha_1}\left(\frac{\beta_2+iav}{\beta_2+iv}\right)^{\alpha_2}, 0<a<1. 
\label{eq:chf:rem:bgam}
\end{equation} 
It means that the \arem\ of a $\bgam$ law with parameters $(\alpha_1,\beta_1, \alpha_2,\beta_2)$ can be seen as the difference $Z^{(u)}_a - Z^{(d)}_a$ of two independent \rv's where $Z^{(u)}_a$ and $Z^{(d)}_a$ are distributed according to the laws $\gampp(\alpha_1, \beta_1)$ and $\gampp(\alpha_2, \beta_2)$, respectively (with the same $a$). 

Now consider a \emph{BDLP} being the difference of two independent compound Poisson processes with exponential jumps $Z(t)=\sum_{n=1}^{N_1(t)}U_n - \sum_{m=1}^{N_2(t)}D_n$. $N_1(t)$ and $N_2(t)$ are two independent Poisson processes with intensities $\lambda_1$ and $\lambda_2$, respectively whereas $U_n\sim\erl_1(\beta_1)$ and $D_m\sim\bar{\erl}_1(\beta_2)$. It is easy to verify that the \chf\ of a process $X(t)$ solution of\myref{eq:genOU_sde} is simply the product of the \chf's of two independent \gou\ processes with parameters $(k, \lambda_1, \beta_1)$ and $(k, \lambda_2, \beta_2)$, respectively.
The stationary law is simply recovered for $t\rightarrow +\infty$ and coincides with a $\bgam$ law.
Once again, as in the case of a \gou\ process, the \chf\ of a \bgou\ process at time $t$ is that of the \arem\ of a $\bgam(\lambda_1/k, \beta_1, \lambda_2/k, \beta_2)$ law (dubbed $\bgampp(\lambda_1/k, \beta_1, \lambda_2/k, \beta_2)$ law) plus a constant $a x_0$ when we take $a=e^{-kt}$. 

Knowing that the $n$th cumulant $\kappa_n(X)$ of the difference $X$ of two independent \rv's $X^{(u)}$ and $X^{(d)}$ is $\kappa_n(X^{(u)})+(-1)^n\kappa_n(X^{(d)})$, after some algebra we find
\begin{eqnarray}
\EXP{Z_a} &=& (1-a)\left(\frac{\alpha_1}{\beta_1} - \frac{\alpha_2}{\beta_2}\right)\\
\VAR{Z_a} &=& (1-a^2) \left(\frac{\alpha_1}{\beta_1^2}+\frac{\alpha_2}{\beta_2^2}\right)\\
\SK{Z_a} &=&\frac{1-a^3}{(1-a^2)^{3/2}}\times\frac{2(\alpha_1\beta_2^3-\alpha_2\beta_1^3)}{(\alpha_1\beta_2^2-\alpha_2\beta_1^2)^{3/2}}\\
\KUR{Z_a} &=&\frac{1+a^2}{1-a^2}\times\frac{6(\alpha_1\beta_2^4+\alpha_2\beta_1^4)}{(\alpha_1\beta_2^2+\alpha_2\beta_1^2)^2}+3.
\label{eq:mom:bgou}
\end{eqnarray} 
moreover, $\EXP{X(t)}=ax_0 + \EXP{Z_a(t)}$ and the variance, the skewness and kurtosis of $X(t)$ and $Z_a$ coincide. 
The distributional properties of a $\bgampp$ law are summarized by the following theorem.
\begin{thm}\label{th:bgamma}
The \chf\ $\chi_a(u, \alpha_1, \beta_1, \alpha_2, \beta_2)$ and the \pdf\ $g_a(x, \alpha_1, \beta_1, \alpha_2, \beta_2)$ of the $\bgampp(\alpha_1, \beta_1, \alpha_2, \beta_2)$ law are
\begin{equation}
\chi_a(v, \alpha_1, \beta_1, \alpha_2, \beta_2)=\sum_{n,m=0}^\infty b_n(a, \alpha_1)b_m(a, \alpha_2)\varphi_u^{n}(v)\varphi_d^{m}(v)
\label{eq:chf:arem:bgam:conv}
\end{equation}
\begin{eqnarray}
g_a(x, \alpha_1, \beta_1, \alpha_2, \beta_2)&=&a^{\alpha_1+\alpha_2}\delta(x)+
a^{\alpha_1}(1-a)f_{n,^{\beta_1}/_a}(x) + a^{\alpha_2}(1-a)f_{m,^{\beta_2}/_a}(-x)+\nonumber\\
&&\sum_{n,m=1}^\infty b_n(a, \alpha_1)b_m(a, \alpha_2)f_{n, m,^{\beta_1}/_a, ^{\beta_2}/_a}(x),
\label{eq:pdf:arem:bgam:conv}
\end{eqnarray}
with
\begin{equation*}
	b_n(a, \alpha) = \binom{\alpha+n-1}{n}a^{\alpha}(1-a)^{n}
\end{equation*}
and
\begin{equation}\label{eq:erl:diff:pdf}                
	f_{n,m,\beta_1,\beta_2}(x)=\left\{
					 \begin{array}{c}
						  \frac{\beta_2 e^{\beta_2 x}}{(n-1)!}\left(\frac{\beta_1}{\beta_1+\beta_2}\right)^{n}\sum_{i=0}^{m-1}\frac{(n+m-i-2)!}{i!(m-i-1)!}
							\times\\
							\left(\frac{\beta_2}{\beta_1+\beta_2}\right)^{m-i-1}(-\beta_2 x)^{i} \quad x<0, \\
						 \frac{\beta_1 e^{-\beta_1 x}}{(m-1)!}\left(\frac{\beta_2}{\beta_1+\beta_2}\right)^{m}\sum_{j=0}^{n-1}\frac{(n+m-j-2)!}{\ell!(n-j-1)!}
						\times\\
						\left(\frac{\beta_1}{\beta_1+\beta_2}\right)^{n-j-1}(\beta_1 x)^{j} \quad x\ge 0
					 \end{array}
				 \right.,
\end{equation}
where
$f_{n,m,\beta_1,\beta_2}(x)$ represents the \pdf\ of the difference $E_u-E_d$ of two independent Erlang distributed rv's $E_u\sim\erl_n(\beta_1)$ and $E_d\sim\erl_d(\beta_2)$, respectively (see Simon \cite{Simon06} page 28).
\end{thm}
\begin{proof}
As already observed, the law $\gampp(\alpha_1,\beta_1, \alpha_2,\beta_2)$ is the law of the difference $Z^{(u)}_a - Z^{(d)}_a$ of two independent \rv's $Z^{(u)}_a$ and $Z^{(d)}_a$ distributed according to the $\gampp(\alpha_1, \beta_1)$ and  $\gampp(\alpha_2, \beta_2)$, respectively. Hence, the \chf\ in\refeq{eq:chf:arem:bgam:conv} is a simple consequence of Theorem \ref{th:arem:density}.

On the other hand, the distributions of $Z^{(u)}_a$ and $Z^{(d)}_a$ can also be considered as two independent Erlang laws $\erl_{S_u}(\beta_1)$ and $\erl_{S_d}(\beta_2)$ with two independent Polya $\pol(\alpha_1,1-a)$ and $\pol(\alpha_2,1-a)$ distributed random indexes $S_u$ and $S_d$, respectively.
Hence, given $S_u=n$ and $S_d=m$ the distribution can be seen as the law of the difference of two independent Erlang \rv's whose \pdf\refeq{eq:erl:diff:pdf} is known in closed form. Combining all these observation leads to the conclusion that the \pdf\ of $Z_a$ is given by\refeq{eq:pdf:arem:bgam:conv}.
\end{proof}
\begin{cor}
The transition density $p(x, t+s| y, s) $of a \bgou\ law with parameters $\alpha_1=\lambda_1/k$, $\beta_1$, $\alpha_2=\lambda_2/k$ and $\beta_2$  is
\begin{equation}
p(x, t+s| y, s) = g_a\left(x-ay, \frac{\lambda_1}{k}, \beta_1, \frac{\lambda_2}{k}, \beta_2\right), \quad a=e^{-kt}. 
\label{eq:bgou:tdensity}
\end{equation}
where $g_a\left(x-ay, \frac{\lambda_1}{k}, \beta_1, \frac{\lambda_2}{k}, \beta_2\right)$ is defined in\refeq{eq:pdf:arem:bgam:conv}.
\end{cor} 	

An interesting case is when the \emph{BDLP} $Z(t)$ is a compound Poisson whose jumps are now distributed according to a double exponential law that is mixture of a positive exponential \rv\ $U\sim\erl_1(\beta_1)$ and a negative exponential \rv\ $D\sim\bar{\erl}_1(\beta_2)$ with mixture parameters $p$ and $q=1-p$ with the following \pdf\ and \chf
\begin{equation}
	f_{\beta_1,\beta_2, p}(x)=p\beta_1e^{-\beta_1 x}\mathds{1}_{x\ge 0} + (1-p)\beta_2e^{\beta_2 x}\mathds{1}_{x< 0}
\label{eq:double:exp:pdf}
\end{equation} 
\begin{equation}\label{eq:double:exp:chf}  
\varphi_{\beta_1,\beta_2, p}(v)  =
  p\frac{\beta_1}{\beta_1-iv} + (1-p) \frac{\beta_2}{\beta_2+iv} = p\varphi_u(v) + (1-p)\varphi_d(v).
\end{equation}
\begin{thm}\label{th:chf:bgou}
Let $X(t)$ be the solution  of\refeq{eq:genOU_sde} where the \emph{BDLP} is a compound Poisson whose jumps are distributed according to the law with \pdf\ and \chf\ in\refeq{eq:double:exp:chf} and\refeq{eq:double:exp:pdf}, respectively, then the \chf\ of $X(t+s)$ conditional on $X(s)$ is given by 
\begin{equation}
\EXP{e^{ivX(t+s)|X(s)}}=e^{ivX(s)e^{-kt}}\times \left(\frac{\beta_1 - ive^{-kt}}{\beta_1 - iv}\right)^{\frac{p\lambda}{k}} \times 
\left(\frac{\beta_2 + ive^{-kt}}{\beta_2 + iv}\right)^{\frac{(1-p) \lambda}{k}} 
\label{eq:chfbgou}
\end{equation}
\end{thm}
\begin{proof}
Based on the results of Dassios and Jang \cite{DJ03} and \cite{Kluge2006},  the logarithm of \chf\ of $X(t+s)$ conditional on $X(s)$ is given by 
				\begin{equation*}
				\log\EXP{e^{ivX(t+s)|X(s)}} = ivX(s)e^{-kt} + \lambda\int_0^t\left(\varphi_J(ve^{-kw}) - 1\right)dw,
				\end{equation*}
				where $\varphi_J(v)$ is the \chf\ of the double exponential in\refeq{eq:double:exp:chf}, therefore we have
				\begin{eqnarray*}
				\log\EXP{e^{ivX(t+s)|X(s)}} &=& ivX(s)e^{-kt} + p\lambda\int_0^t\left(\varphi_u(ve^{-kw}) - 1\right)dw + \\
				&&(1-p)\lambda\int_0^t\left(\varphi_d(ve^{-kw}) - 1\right)dw,
				\end{eqnarray*}
				hence solving the integrals in the second and third terms we have
				\begin{equation}\label{eq:chf:ou:double:exp}
				\EXP{e^{ivX(t+s)|X(s)}}=e^{iuX(s)e^{-kt}}\times\left(\frac{\beta_1 - ive^{-k t}}{\beta_1 -iv}\right)^{\frac{p\lambda}{k}}\left(\frac{\beta_2 + ive^{-k t}}{\beta_2 +iv}\right)^{\frac{(1-p)\lambda}{k}},
				\end{equation}
				that concludes the proof.
\end{proof}
The stationary law is simply recovered for $t\rightarrow +\infty$ and coincides with a $\bgam$ law as summarized by the following corollary.
\begin{cor}
The stationary law of $X(t)$ is a $\bgam$ law with parameters $\alpha_1=p\lambda/k$, $\beta_1$, $\alpha_2=(1-p)\lambda/k$ and $\beta_2$ with $0<p<1$.
\end{cor}
Once again, the \chf\ of a \bgou\ process at time $t$ is that of the \arem\ of a $\bgam(p\lambda/k, \beta_1, (1-p)\lambda/k, \beta_2)$ law plus a constant $a x_0$ when we take $a=e^{-kt}$. 
\subsection{Symmetric $\bgam$\label{subsec:simgou}}
The results of the previous subsection simplify when the \bgou\ process has symmetric parameters where the stationary law is a symmetric bilateral gamma. 
In this case the \emph{BDLP} coincides with a compound Poisson whose jumps are distributed according to a centered Laplace law. 
A simple consequence of Theorem \ref{th:bgamma} and Theorem \ref{th:chf:bgou} with $p=1/2$, $\beta_1=\beta_2$ is the following corollary
\begin{cor}\label{cor:sbgou}
	The \chf\ and the \pdf\ of the \arem\ of a symmetric $\bgam(\alpha, \beta)$ are
	\begin{equation*}
		\chi_a(u) = \sum_{k=0}^\infty\binom{\alpha+k-1}{k} a^{2\alpha}\left(1-a^2\right)^k\left(\frac{\beta^2}{\beta^2+a^2u^2}\right)^k
	\end{equation*}
	\begin{equation*}
	g_a(x, \alpha, \beta)=a^{2\alpha}\delta(x)+\sum_{k=1}^\infty \binom{\alpha+k-1}{k} a^{2\alpha}\left(1-a^2\right)^k\bar{f}_{n,\beta/_a}(x)
	\end{equation*}
	where
	\begin{equation*}
		\bar{f}_{n,\beta}(x) = \frac{\beta}{2^n(n-1)!}\left(\beta|x|\right)^{n-1}e^{-\beta |x|}\sum_{k=1}^{n-1}\frac{(n-1+k)!}{k!(n-1-k)!(2\beta)^k|x|^k}.
	\end{equation*}
\end{cor}
Once more an the law of the \arem\ is an infinite Polya $\pol(\alpha,1-a^2)$-weighted
mixture of bilateral Erlang laws with parameter $\beta/a$.

Hence, taking $a=e^{-kt}$ and
$\alpha=\frac{\lambda}{2k}$, the law of a symmetric \bgou\ at time $t$ coincides with the \chf\
of the \arem\ law of a gamma difference whose transition density $p(x, t+s| y, s)=g_a(x-ay, \frac{\lambda}{2k}, \beta)$. 
In addition,
\begin{eqnarray}
\EXP{Z_a} &=& 0 \\
\VAR{Z_a} &=& (1-a^2) \times\frac{2\alpha}{\beta}\\
\SK{Z_a} &=& 0 \\
\KUR{Z_a} &=&\frac{1+a^2}{1-a^2}\times\frac{3}{\alpha}+3,
\label{eq:mom:bgou:symm}
\end{eqnarray} 
$\EXP{X(t)}=ax_0$ while the variance, the skewness and kurtosis of $X(t)$ and $Z_a$ coincide because these quantities are translation invariant. 
Finally, we remark that for $\alpha=n\in\mathbf{N^*}$ it is straightforward to extend Theorem \ref{th:arem:bin:density} and to represent the \arem\ of a symmetric $\bgam$ as a binomial mixture of bilateral Erlang laws. It suffices to replace $\beta/(\beta-iu)$ in\refeq{eq:bin:erl:chf} with $\beta^2/(\beta^2+u^2)$ and $f_{n,\beta}(x)$ with $\bar{f}_{n,\beta}(x)$ in\refeq{eq:mix:bin}. Finally, the representation based on the generalized binomial theorem at the end of subsection \ref{subsec:erlangRV} can also be extended to the case of symmetric $\bgam$ laws replacing $\omega_k(a, \alpha$) with $\omega_k(a^2, \alpha)$ under the constrain $\frac{1}{\sqrt{2}}\le a <1$. 

\subsection{Simulation Algorithms\label{subsec:simulation:bgou}}
We have seen that the law of the \arem\  $\bgampp(\beta_1, \alpha_1, \beta_2, \alpha_2)$ coincides with that of the difference of the independent \arem's  $\gampp(\beta_1, \alpha_1)$ and $\gampp(\beta_2, \alpha_2)$, respectively. We have also observed that such a distribution coincides with the law at time $t$ of the \bgou$(k,\lambda_1, \beta_1, \lambda_2, \beta_2)$ process if one sets $\alpha_1=\lambda_1/k$ $\alpha_2=\lambda_2/k$ and $a=e^{-kt}$. Based on Theorems \ref{th:chf:gou} and \ref{th:chf:gou:qdz} then, the simulation of the increment of such a process consists of nothing less than implementing the algorithms detailed in section \ref{subsec:simulation:gou} two times. To this end, for sake of brevity, the detailed steps are not repeated here.  

In contrast, we here detail some simulation algorithms tailored to the symmetric case. For instance, because of Corollary \ref{cor:sbgou}, the implementation steps of Algorithm \ref{alg:sd} can be replaced by the following ones.
\begin{algorithm}
\caption{ }\label{alg:bgousd}
		\begin{algorithmic}[1]
		\For{ $m=1, \dots, M$}
		\State $\alpha\gets\lambda/2k,\;\; a\gets e^{-k\Delta t_m}$
		\State $b\gets B\sim \pol(\alpha,1-a^2)$ \Comment{Generate a Polya $(\alpha,1-a^2)$ \rv}
		\State $z_a^{(r)} \gets Z_a^{(r)}\sim\erl_{b}\left(^{\beta}/_a\right), r\in\{u, d\}$; \Comment{ Generate two independent  Erlang \rv's with rate $\beta/	a$}
		\State $z_a^m = z_a^{(u)} - z_a^{(d)}$
		\State $X(t_m)\gets a\,X(t_{m-1}) + z_a^{(m)}$.
		\EndFor
		\end{algorithmic}
\end{algorithm}

In addition, knowing the density in closed-form 
\begin{equation}\label{eq:accRej:bgou}
	g_a(x) \le \sum_{k\ge 0}^{\infty}\omega_k(a^2, \alpha)^+ \bar{f}_{k, \beta}(x) = \overline{g}(x) = c g(x)
\end{equation}
where
\begin{equation}\label{eq:mix}
   1<c =\sum_{k\ge 0}^{\infty}\omega_k(a^2, \alpha)^+<\infty\qquad p_k =\frac{\omega_k(a^2, \alpha)^+}{c} \qquad g(x)=\sum_{k\ge 0}^{\infty}p_k \bar{f}_{k,
   \beta}(x)
\end{equation}
 so that $g(x)$ turns out to be a true mixture of symmetric bilateral Erlang laws,
namely the \pdf\ of
\begin{equation*}
    V=\sum_{i=0}^{S}\left(X_i^{(u)}-X_i^{(d)}\right)\quad X_i^{(r)}\sim\erl_1(\beta)\quad\PR{S = k} = p_k\quad r\in\{u,d\},
\end{equation*}
one can adapt Algorithm \ref{alg:sd:gamma} to the case of a symmetric \bgou\ process simply replacing the fourth step with Algorithm \ref{alg:sd:laplace} and using the \pdf's in\refeq{eq:accRej:bgou} and\refeq{eq:mix}.
\begin{algorithm}
\caption{($\frac{1}{\sqrt{2}}\le a <1$)}\label{alg:sd:laplace}
		\begin{algorithmic}[1]		
		\setcounter{ALG@line}{3}
		\State $\bar{z}^{(r)}\gets\bar{Z}^{(r)}\sim\erl(S,1), r\in\{u,d\}$\Comment{Generate two independent standard Erlang \rv's.}	
		\State $\bar{z} \gets \bar{z}^{(u)}-\bar{z}^{(d)}$
		\end{algorithmic}
\end{algorithm}

In addition, Algorithm \ref{alg1:ref} can also be extended simply substituting the sixth step by those here below.
 \begin{algorithm}
 \caption{ }\label{alg1:ref:bgou}
		\begin{algorithmic}[1]
		\setcounter{ALG@line}{5}
		\State Generate $n$ \iid\ $J^{(r)}_i\sim\erl_1(\beta), i=1,\dots, n, r\in\{u, d\}$, \Comment {Generate two sets of independent exponential \rv's with random rate $\beta}$
		\State $J_i\gets J^{(u)}_i - J^{(d)}_i$		
		\end{algorithmic}
\end{algorithm}

Finally, the following theorem extends the approach in Qu et al. \cite{QDZ19} to the case of a symmetric \bgou\ process avoiding then to run Algorithm \ref{alg:qdz} twice. \begin{thm}
\begin{equation}
\EXP{e^{iuX(t+s)}|X(s)}=e^{iuX(s)e^{-kt}}\times e^{\lambda t \left(\varphi_{\tilde{L}}(u)-1\right)}
\end{equation}
where
\begin{equation}
\varphi_{\tilde{L}}(u)=\int_0^1\frac{\beta^2e^{2ktv}}{\beta^2e^{2ktv} + u^2}dv
\label{eq:chf:bgou:qdz}
\end{equation}
\end{thm}
\noindent the right-hand side in\refeq{eq:chf:bgou:qdz} is then the \chf\ of compound Poisson whose jumps are independent copies $\tilde{J}_i$ distributed according to a uniform mixture of centered Laplace laws with random parameter $\beta e^{ktU}$ with $U\sim\unif([0,1])$. 
\begin{proof}
From Theorem \ref{th:chf:gou} and Theorem \ref{th:chf:gou:qdz} we know that   
\begin{equation*}
e^{\frac{\lambda t}{2}\left(\varphi_{\tilde{J}}(u)-1\right)}=\left(\frac{\beta -iue^{-kt}}{\beta -iu}\right)^{\frac{\lambda}{2\,k}}
\end{equation*}
where $\varphi_{\tilde{J}}(u)$ is defined in\refeq{eq:chf:exp:qdz} then from Theorem \ref{th:bgamma} and Theorem \ref{th:chf:bgou} with $p=1/2$, $\beta_1=\beta_2$, we have
\begin{equation*}
\EXP{e^{iuX(t+s)}|X(s)}=e^{iuX(s)e^{-kt}}\times e^{\frac{\lambda t}{2}t \left(\varphi_{\tilde{J}}(u)+ \varphi_{\tilde{J}}(-u)-2\right)} =
e^{iuX(s)e^{-kt}}\times e^{\lambda t \left(\frac{\varphi_{\tilde{J}}(u)+ \varphi_{\tilde{J}}(-u)}{2}	-1\right)}
\end{equation*}
on the other hand, we observe that
\begin{equation*}
\frac{\varphi_{\tilde{J}}(u)+ \varphi_{\tilde{J}}(-u)}{2} =
\frac{1}{2}\int_0^1\left(\frac{\beta e^{ktv}}{\beta e^{tv} - iu}+\frac{\beta e^{ktv}}{\beta e^{ktv} + iu}\right)dv =\int_0^1\frac{\beta^2e^{ktv}}{\beta ^2e^{ktv} + u^2}dv
\end{equation*}
that concludes the proof.
\end{proof}
\noindent It turns out that a symmetric \bgou\ can be simulated as detailed in Algorithm \ref{alg:bgou:qdz}
\begin{algorithm}
\caption{ }\label{alg:bgou:qdz}
\begin{algorithmic}[1]
		\For{ $m=1, \dots, M$}
		\State Generate $n\sim\poiss(\lambda\Delta t_m)$,\Comment {Poisson \rv~with intensity $\lambda\Delta t_m$}
		\State Generate $n$ \iid\ uniform \rv's $\bm{u}=(u_1, \dots, u_n)\,\sim\,\unif([0,1]^n)$.
		\State $\beta_i^{(r)}\gets\beta  e^{k\Delta t_m u_i}, i=1,\dots, n, r\in\{u, d\}$.
		\State Generate $n$ \iid\ $\tilde{J}^{(r)}_i\sim\erl_1(\beta_i^{(r)}), i=1,\dots, n$, \Comment {Generate two sets of independent exponential \rv's with random rate $\beta_i^{(r)}}$
		\State $\tilde{J}_i\gets \tilde{J}^{(u)}_i - \tilde{J}^{(d)}_i$		
		\State $X(t_m)\gets X(t_{m-1})e^{-k\Delta t_m} + \sum_{i=1}^{n}\tilde{J}_i$.
		\EndFor
		\end{algorithmic}
\end{algorithm}


\section{Simulation Experiments}\label{sec:sim:exper}
In this section we compare the performance of the Algorithms detailed in subsection \ref{subsec:simulation:gou} for the \gou\ process and in subsection \ref{subsec:simulation:bgou} for the \bgou\ process. The performance is ranked in terms of convergence and in terms of CPU times.
All the simulation experiments in the present paper
have been conducted using \emph{MATLAB R2019a} with a $64$-bit
Intel Core i5-6300U CPU, 8GB \footnote{The relative codes are available at \url{https://github.com/piergiacomo75/GammaOUBiGammaOU} }.  
As an additional validation, the
comparisons of the simulation computational times have
also been performed with \emph{R} and \emph{Python}  leading to the same conclusions.

We first consider a \gou\ process with parameters $(k, \lambda, \beta, x_0) =(36, 10, 3, 0)$ and we only simulate one time step at $\Delta t=1/365$. We observe that Algorithm \ref{alg:sd} is still suitable because $0.5\le a <1$  ($a=e^{-k\Delta t}\approx 0.9061$) where we have truncated the series in\refeq{eq:accRej} and\refeq{mix} at the $40$-th term. Here we have chosen different values for the Poisson intensity $\lambda$ and the rate of the jump size $\beta$ to let the investigation be more thorough. 

In realistic examples, one could estimate the parameters relying on the closed form of the transition densities of the process, using the generalized method of moments or the least squares method. The idea of coupling a \gou\ process with a Gaussian-OU process is common in the modeling of energy prices (see among others for instance, Cartea and Figueroa \cite{CarteaFigueroa}, Kjaer \cite{Kjaer2008} and Kluge \cite{Kluge2006}), indeed, the choice of the parameters above is motivated by the fact that these numbers look like realistic values that can be adopted for the
pricing of energy facilities,. Beyond the energy world, applications of the \gou\ process to portfolio selection or to credit risk can be found in Schoutens and Cariboni \cite{SchoutensCariboni}, Bianchi and Fabozzi \cite{Bianchi2015} and Bianchi and Tassinari \cite{BianchiTassinari}. 
		\begin{table}[ht!]
    \centering\scriptsize
		\resizebox{\textwidth}{!}{
        \begin{tabular}{*{10}{|cc|cc|cc|cc|cc}}
					\hline
										\multicolumn{2}{|c}{} &  \multicolumn{2}{|c}{$\EXP{X(t)} = 0.0087$} & \multicolumn{2}{|c}{$\VAR{X(t)}=0.0055$} & \multicolumn{2}{|c}{$\SK{X(t)}=12.83$} & \multicolumn{2}{|c|}{$\KUR{X(t)}=222.71$}\\			
					\hline	
					\multicolumn{10}{|c|}{Algorithm\myref{alg:sd}}\\					
					\hline
					$N_S$ & CPU & MC & error \% & MC & error \% & MC & error \% & MC & error \%	\\							
					\hline
					$10000$ & $0.0050$ & $0.0089$ & $2.12$\% & $0.0052$ & $6.42$\% & $11.25$ & $12.30$\% & $157.64$ & $27.25$\% \\
					$40000$ & $0.0136$ & $0.0091$ & $4.27$\% & $0.0060$ & $8.89$\% & $12.41$ & $3.24$\% & $198.06$ & $8.60$\% \\
					$160000$ & $0.0524$ & $0.0084$ & $3.76$\% & $0.0053$ & $4.92$\% & $13.01$ & $1.44$\% & $227.59$ & $5.02$\% \\
					$640000$ & $0.2139$ & $0.0086$ & $1.07$\% & $0.0054$ & $1.95$\% & $12.83$ & $0.01$\% & $220.47$ & $1.74$\% \\
					$2560000$ & $0.8757$ & $0.0088$ & $1.42$\% & $0.0056$ & $1.81$\% & $12.69$ & $1.09$\% & $220.91$ & $1.94$\% \\
					\hline
					\hline
					\multicolumn{10}{|c|}{Algorithm\myref{alg:sd:gamma}}\\
					\hline
					$10000$ & $0.0718$ & $0.0083$ & $4.70$\% & $0.0049$ & $12.22$\% & $12.04$ & $6.16$\% & $189.14$ & $12.72$\%\\
					$40000$ & $0.1948$ & $0.0093$ & $6.71$\% & $0.0062$ & $11.60$\% & $12.67$ & $1.25$\% & $209.80$ & $3.19$\%\\
					$160000$ & $0.6349$ & $0.0092$ & $5.53$\% & $0.0060$ & $8.49$\% & $12.48$ & $2.72$\% & $205.54$ & $5.15$\%\\
					$640000$ & $2.3906$ & $0.0088$ & $0.63$\% & $0.0055$ & $0.11$\% & $12.69$ & $1.08$\% & $216.85$ & $0.07$\%\\
					$2560000$ & $8.9852$ & $0.0087$ & $0.19$\% & $0.0055$ & $0.07$\% & $12.74$ & $0.71$\% & $217.53$ & $0.38$\%\\
					\hline
					\hline
					\multicolumn{10}{|c|}{Algorithm\myref{alg1:ref}}\\
					\hline
					$10000$ & $0.15$ & $0.0077$ & $11.23$\% & $0.0044$ & $20.21$\% & $13.46$ & $4.93$\% & $248.17$ & $14.52$\%\\
					$40000$ & $0.48$ & $0.0085$ & $2.77$\% & $0.0050$ & $9.54$\% & $11.92$ & $7.08$\% & $182.57$ & $15.75$\%\\
					$160000$ & $1.81$ & $0.0084$ & $3.18$\% & $0.0053$ & $4.46$\% & $13.18$ & $2.72$\% & $237.77$ & $9.72$\%\\
					$640000$ & $7.26$ & $0.0087$ & $0.01$\% & $0.0055$ & $0.05$\% & $12.79$ & $0.30$\% & $220.48$ & $1.74$\%\\
					$2560000$ & $29.98$ & $0.0086$ & $0.64$\% & $0.0054$ & $1.45$\% & $12.83$ & $0.01$\% & $222.45$ & $2.65$\%\\
					\hline
					\hline
					\multicolumn{10}{|c|}{Algorithm\myref{alg:qdz}}\\
					\hline
					$10000$ & $0.19$ & $0.0090$ & $3.33$\% & $0.0056$ & $0.94$\% & $11.45$ & $10.71$\% & $160.49$ & $25.94$\%\\
					$40000$ & $0.59$ & $0.0090$ & $3.93$\% & $0.0059$ & $6.70$\% & $12.28$ & $4.24$\% & $191.69$ & $11.54$\%\\
					$160000$ & $2.35$ & $0.0085$ & $1.79$\% & $0.0053$ & $3.66$\% & $12.83$ & $0.02$\% & $224.77$ & $3.72$\%\\
					$640000$ & $9.32$ & $0.0086$ & $0.75$\% & $0.0056$ & $0.52$\% & $13.23$ & $3.17$\% & $220.55$ & $1.77$\%\\
					$2560000$ & $37.32$ & $0.0088$ & $0.65$\% & $0.0056$ & $0.73$\% & $12.71$ & $0.89$\% & $217.25$ & $0.25$\%\\
					\hline
        \end{tabular}				
		}
    \scriptsize
    \caption{\footnotesize{CPU times in seconds and comparison  among the true $\EXP{X(t)}$, $\VAR{X(t)}$, $\SK{X(t)}$ and $\KUR{X(t)}$ of a \gou\ process at time $t=1/365$ with $(k, \lambda, \beta, x_0) =(36, 10, 3, 0)$ and their relative estimated values with $N_S$ MC scenarios using Algorithms \myref{alg:sd}, \ref{alg:sd:gamma}, \ref{alg1:ref} and \ref{alg:qdz}.}}\label{tab:gou:one:step:MC}
\end{table}

		\begin{table}[ht!]
    \centering\scriptsize
		\resizebox{\textwidth}{!}{
        \begin{tabular}{*{10}{|cc|cc|cc|cc|cc}}
					\hline
										\multicolumn{2}{|c}{} &  \multicolumn{2}{|c}{$\EXP{X(t)} = 6.8522$} & \multicolumn{2}{|c}{$\VAR{X(t)}=1.2642$} & \multicolumn{2}{|c}{$\SK{X(t)}=2.1861$} & \multicolumn{2}{|c|}{$\KUR{X(t)}=9.4919$}\\			
					\hline	
					\multicolumn{10}{|c|}{Algorithm\myref{alg:sd}}\\					
					\hline
					$N_S$ & CPU & MC & error \% & MC & error \% & MC & error \% & MC & error \%	\\							
					\hline
					$10000$ &$0.0203$ &$6.8612$ &$0.13$\% &$1.2922$ &$2.17$\% &$2.1615$ &$-1.14$\% &$9.0429$ &$-4.97$\%\\
					$40000$ &$0.0262$ &$6.8580$ &$0.08$\% &$1.3048$ &$3.11$\% &$2.2173$ &$1.41$\% &$9.6260$ &$1.39$\%\\
					$160000$ &$0.1752$ &$6.8476$ &$-0.07$\% &$1.2647$ &$0.04$\% &$2.2338$ &$2.14$\% &$9.9624$ &$4.72$\%\\
					$640000$ &$0.7359$ &$6.8531$ &$0.01$\% &$1.2680$ &$0.30$\% &$2.1955$ &$0.43$\% &$9.6146$ &$1.28$\%\\
					$2560000$ &$3.1980$ &$6.8520$ &$0.00$\% &$1.2625$ &$-0.13$\% &$2.1904$ &$0.20$\% &$9.5559$ &$0.67$\%\\
					\hline
					\hline
					\multicolumn{10}{|c|}{Algorithm\myref{alg:sd:gamma}}\\
					\hline
					$10000$ &$0.238$ &$6.8358$ &$-0.24$\% &$1.1968$ &$-5.63$\% &$2.1230$ &$-2.97$\% &$8.7994$ &$-7.87$\%\\
					$40000$ &$0.721$ &$6.8537$ &$0.02$\% &$1.2628$ &$-0.11$\% &$2.1580$ &$-1.30$\% &$9.3910$ &$-1.07$\%\\
					$160000$ &$2.531$ &$6.8549$ &$0.04$\% &$1.2636$ &$-0.05$\% &$2.1902$ &$0.19$\% &$9.6369$ &$1.50$\%\\
					$640000$ &$14.451$ &$6.8526$ &$0.01$\% &$1.2655$ &$0.10$\% &$2.1828$ &$-0.15$\% &$9.4517$ &$-0.43$\%\\
					$2560000$ &$60.933$ &$6.8512$ &$-0.01$\% &$1.2602$ &$-0.32$\% &$2.1903$ &$0.19$\% &$9.5654$ &$0.77$\%\\
					\hline
					\hline
					\multicolumn{10}{|c|}{Algorithm\myref{alg1:ref}}\\
					\hline
					$10000$ &$0.41$ &$6.8664$ &$0.21$\% &$1.2755$ &$0.89$\% &$2.1181$ &$-3.21$\% &$8.9754$ &$-5.76$\%\\
					$40000$ &$1.37$ &$6.8423$ &$-0.14$\% &$1.2279$ &$-2.96$\% &$2.1568$ &$-1.36$\% &$9.2733$ &$-2.36$\%\\
					$160000$ &$5.53$ &$6.8578$ &$0.08$\% &$1.2830$ &$1.47$\% &$2.1674$ &$-0.86$\% &$9.1683$ &$-3.53$\%\\
					$640000$ &$22.09$ &$6.8518$ &$-0.01$\% &$1.2673$ &$0.25$\% &$2.1968$ &$0.49$\% &$9.5983$ &$1.11$\%\\
					$2560000$ &$96.19$ &$6.8523$ &$0.00$\% &$1.2637$ &$-0.04$\% &$2.1847$ &$-0.07$\% &$9.5011$ &$0.10$\%\\
					\hline
					\multicolumn{10}{|c|}{Algorithm\myref{alg:qdz}}\\
					\hline
					$10000$ &$0.44$ &$6.8499$ &$-0.03$\% &$1.2870$ &$1.77$\% &$2.3304$ &$6.19$\% &$10.879$ &$12.75$\%\\
					$40000$ &$1.63$ &$6.8431$ &$-0.13$\% &$1.2270$ &$-3.03$\% &$2.1825$ &$-0.16$\% &$9.5881$ &$1.00$\%\\
					$160000$ &$6.54$ &$6.8515$ &$-0.01$\% &$1.2620$ &$-0.18$\% &$2.1988$ &$0.58$\% &$9.6740$ &$1.88$\%\\
					$640000$ &$29.53$ &$6.8535$ &$0.02$\% &$1.2688$ &$0.36$\% &$2.1898$ &$0.17$\% &$9.5400$ &$0.50$\%\\
					$2560000$ &$100.99$ &$6.8525$ &$0.00$\% &$1.2655$ &$0.10$\% &$2.1871$ &$0.05$\% &$9.4898$ &$-0.02$\%\\
					\hline
        \end{tabular}				
		}
    \scriptsize
    \caption{\footnotesize{CPU times in seconds and comparison among the true $\EXP{X(t)}$, $\VAR{X(t)}$, $\SK{X(t)}$ and $\KUR{X(t)}$ of a \gou\ process at time $t=1$ with $(k, \lambda, \beta, x_0) =(0.5, 1, 1, 10)$ and their relative estimated values with $N_S$ MC scenarios using Algorithms \myref{alg:sd}, \ref{alg:sd:gamma}, \ref{alg1:ref} and \ref{alg:qdz}.} }\label{tab:gou:trajectory:MC}
\end{table}

		\begin{figure}
				\caption{\footnotesize{\gou\ with $(k, \lambda, \beta, x_0) =(36, 10, 3, 0)$, $\Delta t = 1/365.$}}
				\begin{subfigure}[t]{.5\textwidth}{
				\centering										
					\includegraphics[width=70mm]{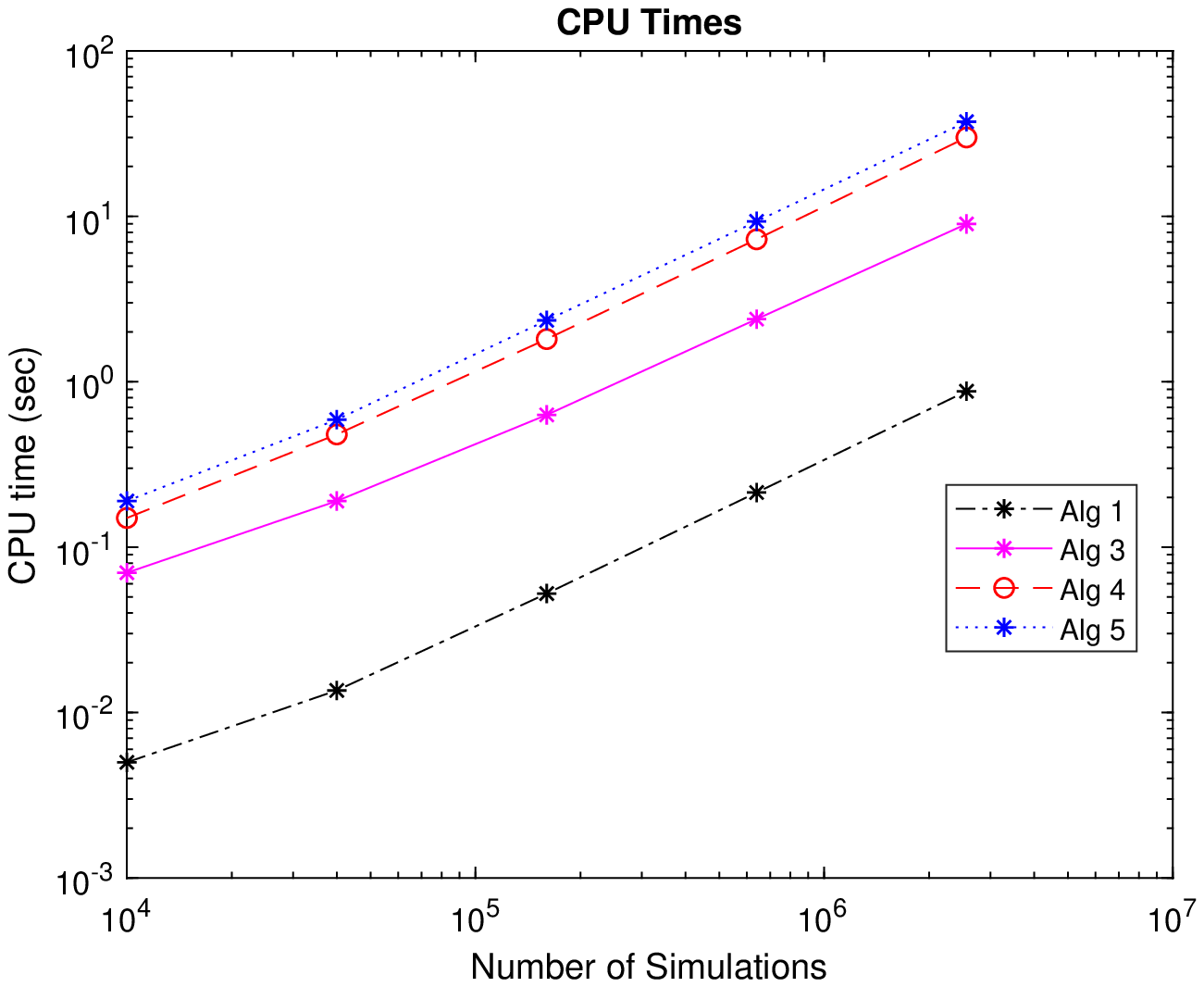}
					\caption{CPU times in seconds.}\label{fig:gou:ComputationalTimes}
					}
				\end{subfigure}
				\begin{subfigure}[t]{.5\textwidth}{
						\centering
					\includegraphics[width=70mm]{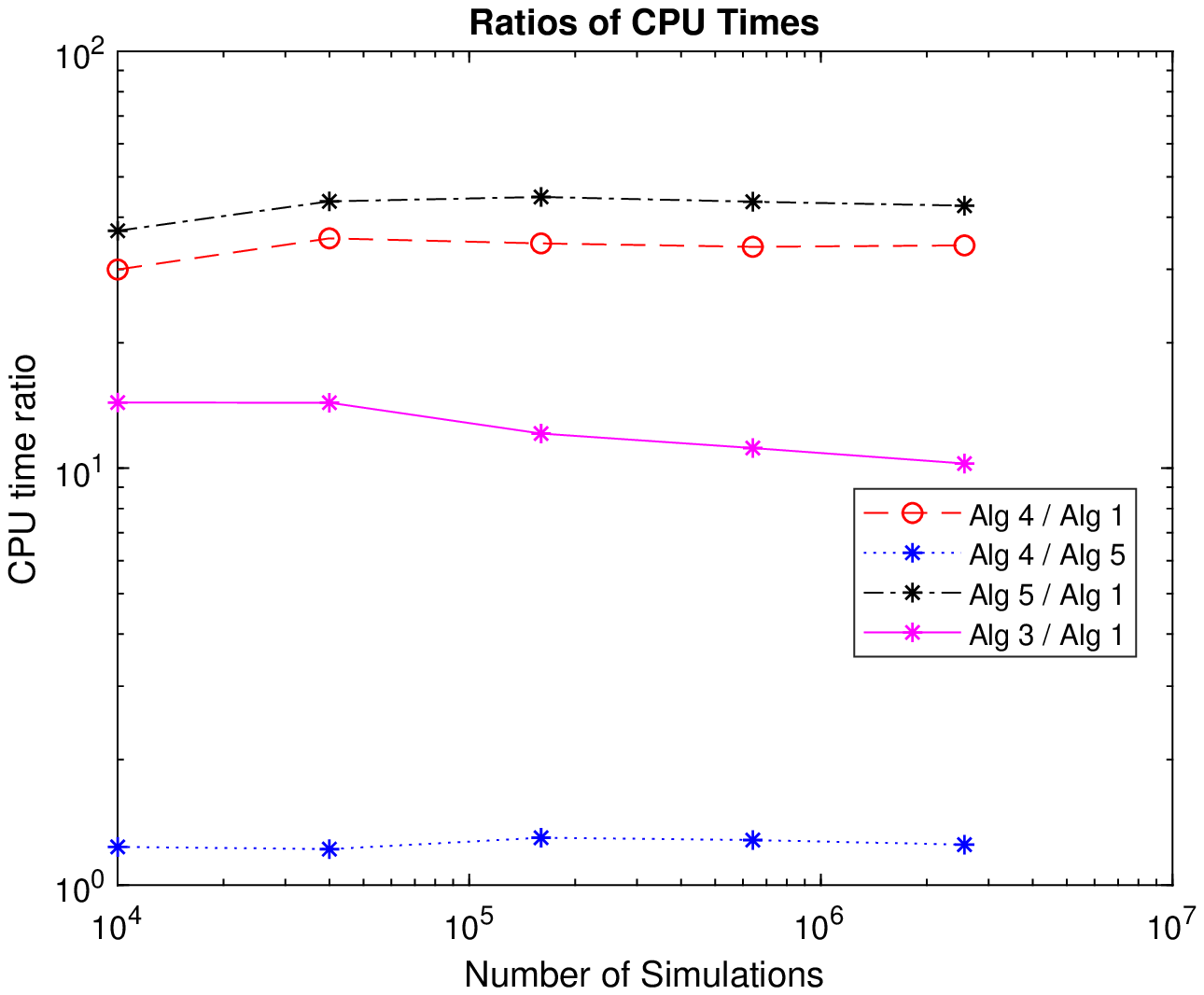}
						\caption{\footnotesize{Ratios CPU times}}\label{fig:gou:RatioComputationalTimes}
					}
				\end{subfigure}
		\end{figure}%
Table \ref{tab:gou:one:step:MC} reports the CPU times in seconds of all the approaches and compares the MC estimated values of the true $\EXP{X(t)}$, $\VAR{X(t)}$, $\SK{X(t)}$ and $\KUR{X(t)}$ at time $t=1/365$. Varying the number of simulations $N_S$, we can conclude that all the algorithms are equally convergent, although it seems that a large number of simulations is required to achieve a good estimate of the kurtosis. On the other hand, their computational performance is quite different. 
Figure~\ref{fig:gou:ComputationalTimes} and~\ref{fig:gou:RatioComputationalTimes} clearly show that Algorithm\myref{alg:sd}  by far outperforms all other approaches. It provides a remarkable improvement in terms of computational time that is at least 
$30$ times smaller than that of any other alternative available in the literature. With our computer generating $N_S=2560000$ values of $X(t)$ at time $t=1/365$ does not even take a second in contrast to several seconds using the other alternatives. Algorithm \ref{alg:sd:gamma} is also faster than Algorithms \ref{alg1:ref} and Algorithm \ref{alg:qdz} although based on a acceptance-rejection method, but unfortunately, it is only applicable under the constraint $0.5\le e^{-k\Delta t}<1$. To conclude, it seems that Algorithms \ref{alg1:ref} and Algorithm \ref{alg:qdz} exhibit similar CPU times. 

Of course, the superior performances of Algorithm \ref{alg:sd} with respect to all the alternatives becomes even more remarkable when the entire trajectory over a time grid is simulated. 
To this end, we generate the skeleton of the process on an equally space time grid $t_1, \dots, t_M$ with $M=4$ and $\Delta t=1/4$. In order to better highlight the difference in performance among the approaches, we have here chosen the same parameter set as in Qu et al. \cite{QDZ19}, $(k, \lambda, \beta, x_0)=(0.5, 1, 1, 10)$. 

The results in Table \ref{tab:gou:trajectory:MC} confirm that our proposal provides the smallest CPU times making Algorithm \ref{alg:sd} very attractive for real-time calculations. We remark that the CPU times in Table \ref{tab:gou:trajectory:MC} are relative to the simulation of entire trajectory with four time steps while instead, the estimated statistics refer to the process at time $t=1$. Refining the time grid with a smaller time step will increase the overall computational times almost linearly making all alternatives to Algorithm \ref{alg:sd} not competitive for real-time applications. 
It is also worthwhile noticing that our implementation, although based on a less powerful computer, returns smaller CPU times than those reported in Qu et al. \cite{QDZ19} relative these authors' approach.
Finally, as described in Section \ref{sec:bgou}, the simulation of a \bgou\ process can be obtained by repeating the algorithms above two times therefore, we can extrapolate the same conclusions with regards to the \bgou\ case. 

We conclude this section illustrating the results of the numerical experiments relative to a symmetric \bgou\ process where we have chosen the same set of parameters selected for the \gou process. $\EXP{X(t)}=x_0 e^{-kt}$ and the skewness is zero therefore in Table \ref{tab:gou:one:step:MC} we show the CPU times in seconds and the MC estimated values of the true $\VAR{X(t)}$ and $\KUR{X(t)}$ at time $t=1/365$ only.  We also remark that Algorithm \ref{alg:sd:laplace} is also applicable because $\sqrt{2}/2\le a <1$ using both parameter sets. The conclusions are very much in line with what found for a \gou\ process. As expected, all the approaches are equally convergent and the CPU times are higher that those for the \gou\ case because all the solutions require additional steps. From Figures~\ref{fig:bgou:ComputationalTimes} and~\ref{fig:bgou:RatioComputationalTimes} one can observe that Algorithm \ref{alg:bgousd} is by far the fastest solution and Algorithm \ref{alg:sd:laplace}, even if based on an acceptance rejection method, is once more a faster solution than  Algorithms \ref{alg1:ref:bgou} and Algorithm \ref{alg:bgou:qdz}. On the other hand, these last two approaches seem to be equally fast with the former slightly outperforming the approach in Qu et al. adapted to the symmetric \bgou\ process. 

In Table \ref{tab:bgou:trajectory:MC} we also report the results of generating the trajectory of a symmetric \bgou\ with the same parameters and time grid of the \gou\ case. The values in Table~\ref{tab:bgou:trajectory:MC} once more  confirm that our newly developed simulation approach, detailed Algorithm \ref{alg:bgousd}, exhibits high accuracy as well as efficiency and in particular, largely outperforms any other alternative. 
				
			\begin{table}[ht!]
    \centering\scriptsize
		\resizebox{\textwidth}{!}{
        \begin{tabular}{{|c|c|cc|cc||c|cc|cc|}}
					\hline
					\multicolumn{2}{|c}{} &  \multicolumn{2}{|c|}{$\VAR{X(t)}=0.0055$} & \multicolumn{2}{c||}{$\KUR{X(t)}=222.71$} & &
					\multicolumn{2}{c}{$\VAR{X(t)}=0.0055$}  & \multicolumn{2}{|c|}{$\KUR{X(t)}=222.71$}\\			
					\hline	
					& \multicolumn{5}{c||}{Algorithm\myref{alg:bgousd}} & \multicolumn{5}{c|}{Algorithm\myref{alg:sd:laplace}}\\					
					\hline
					$N_S$ & CPU & MC & error \% & MC & error \% & CPU & MC & error \% & MC & error \%	\\							
					\hline
					$10000$&$0.0034$&$0.0062$&$12.9$\%&$141.39$&$34.8$\%&$0.0837$&$0.0066$&$20.1$\%&$218.45$&$0.8$\%\\
					$40000$&$0.0121$&$0.0055$&$0.97$\%&$202.53$&$6.54$\%&$0.1634$&$0.0057$&$2.96$\%&$176.28$&$18.66$\%\\
					$160000$&$0.0559$&$0.0054$&$2.83$\%&$207.09$&$4.44$\%&$0.6415$&$0.0059$&$6.20$\%&$227.93$&$5.18$\%\\
					$640000$&$0.2284$&$0.0054$&$2.35$\%&$228.70$&$5.53$\%&$2.3321$&$0.0057$&$2.79$\%&$217.46$&$0.34$\%\\
					$2560000$&$0.9188$&$0.0055$&$1.15$\%&$211.83$&$2.25$\%&$9.2060$&$0.0055$&$0.39$\%&$218.86$&$0.99$\%\\
					\hline
					\hline
					& \multicolumn{5}{c||}{Algorithm\myref{alg1:ref:bgou}} & \multicolumn{5}{c|}{Algorithm\myref{alg:bgou:qdz}}\\					
					\hline
					$10000$&$0.15$&$0.0063$&$13.2$\%&$266.74$&$23.1$\%&$0.1915$&$0.0052$&$6.4$\%&$190.45$&$12.1$\%\\
					$40000$&$0.56$&$0.0052$&$5.14$\%&$203.09$&$6.28$\%&$0.7769$&$0.0059$&$6.94$\%&$212.10$&$2.13$\%\\
					$160000$&$2.24$&$0.0055$&$0.26$\%&$246.67$&$13.82$\%&$3.0793$&$0.0060$&$8.11$\%&$223.25$&$3.02$\%\\
					$640000$&$8.98$&$0.0054$&$1.48$\%&$216.01$&$0.32$\%&$12.3615$&$0.0056$&$0.58$\%&$224.35$&$3.06$\%\\
					$2560000$&$36.16$&$0.0055$&$0.54$\%&$223.38$&$0.52$\%&$49.8659$&$0.0055$&$0.31$\%&$217.83$&$0.52$\%\\
					\hline
        \end{tabular}				
		}
    \scriptsize
    \caption{\footnotesize{CPU times in seconds and comparison among the true $\VAR{X(t)}$ and $\KUR{X(t)}$ of a symmetric \bgou\ process at time $t=1/365$ with $(k, \lambda, \beta, x_0) =(36, 10, 3, 0)$ and their relative estimated values with $N_S$ MC scenarios using Algorithms\myref{alg:bgousd}, \ref{alg:sd:laplace}, \ref{alg1:ref:bgou} and \ref{alg:bgou:qdz}.}}\label{tab:bgou:one:step:MC}
\end{table}

			\begin{table}[ht!]
    \centering\scriptsize
		\resizebox{\textwidth}{!}{
        \begin{tabular}{{|c|c|cc|cc||c|cc|cc|}}
					\hline
					\multicolumn{2}{|c}{} &  \multicolumn{2}{|c}{$\VAR{X(t)}=1.2642$} & \multicolumn{2}{|c||}{$\KUR{X(t)}=9.4919$} & &
					\multicolumn{2}{c}{$\VAR{X(t)}=1.2642$}  & \multicolumn{2}{|c|}{$\KUR{X(t)}=9.4919$}\\			
					\hline	
					& \multicolumn{5}{c||}{Algorithm\myref{alg:bgousd}} & \multicolumn{5}{c|}{Algorithm\myref{alg:sd:laplace}}\\					
					\hline
					$N_S$ & CPU & MC & error \% & MC & error \% & CPU & MC & error \% & MC & error \%	\\							
					\hline
					$10000$&$0.0157$&$1.2539$&$-0.82$\%&$9.306$&$-1.99$\%&$0.2502$&$1.2920$&$2.15$\%&$8.966$&$-5.87$\%\\
					$40000$&$0.0193$&$1.2148$&$-4.07$\%&$9.370$&$-1.31$\%&$0.9680$&$1.2478$&$-1.32$\%&$9.249$&$-2.62$\%\\
					$160000$&$0.1098$&$1.2598$&$-0.35$\%&$9.498$&$0.06$\%&$3.4811$&$1.2619$&$-0.19$\%&$9.450$&$-0.44$\%\\
					$640000$&$0.4614$&$1.2617$&$-0.20$\%&$9.509$&$0.18$\%&$14.713$&$1.2576$&$-0.53$\%&$9.536$&$0.47$\%\\
					$2560000$&$3.1008$&$1.2650$&$0.06$\%&$9.537$&$0.47$\%&$61.448$&$1.2648$&$0.05$\%&$9.509$&$0.18$\%\\
					\hline
					\hline
					& \multicolumn{5}{c||}{Algorithm\myref{alg1:ref:bgou}} & \multicolumn{5}{c|}{Algorithm\myref{alg:bgou:qdz}}\\				
					\hline
					$10000$&$0.6143$&$1.2844$&$1.6$\%&$10.60$&$10.5$\%&$0.7889$&$1.2325$&$-2.57$\%&$8.222$&$-15.4$\%\\
					$40000$&$2.2989$&$1.2950$&$2.38$\%&$9.668$&$1.83$\%&$3.1908$&$1.2771$&$1.01$\%&$9.168$&$-3.54$\%\\
					$160000$&$9.2519$&$1.2578$&$-0.51$\%&$9.510$&$0.19$\%&$12.327$&$1.2564$&$-0.63$\%&$9.582$&$0.94$\%\\
					$640000$&$36.601$&$1.2617$&$-0.20$\%&$9.579$&$0.91$\%&$48.689$&$1.2622$&$-0.16$\%&$9.422$&$-0.74$\%\\
					$2560000$&$147.79$&$1.2628$&$-0.12$\%&$9.455$&$-0.39$\%&$199.71$&$1.2651$&$0.06$\%&$9.498$&$0.06$\%\\
					\hline
        \end{tabular}				
		}
    \scriptsize
    \caption{\footnotesize{CPU times in seconds and comparison among the true $\VAR{X(t)}$ and $\KUR{X(t)}$ of a symmetric \bgou\ process at time $t=1$ with $(k, \lambda, \beta, x_0) =(0.5, 1, 1, 10)$ and their relative estimated values with $N_S$ MC scenarios using Algorithms\myref{alg:bgousd}, \ref{alg:sd:laplace}, \ref{alg1:ref:bgou} and \ref{alg:bgou:qdz}.}}\label{tab:bgou:trajectory:MC}
\end{table}

			\begin{figure}
					\caption{\footnotesize{Symmetric \bgou\ with $(k, \lambda, \beta, x_0) =(36, 10, 3, 0)$, $\Delta t = 1/365.$}}
					\begin{subfigure}[t]{.5\textwidth}{
					\centering										
						\includegraphics[width=70mm]{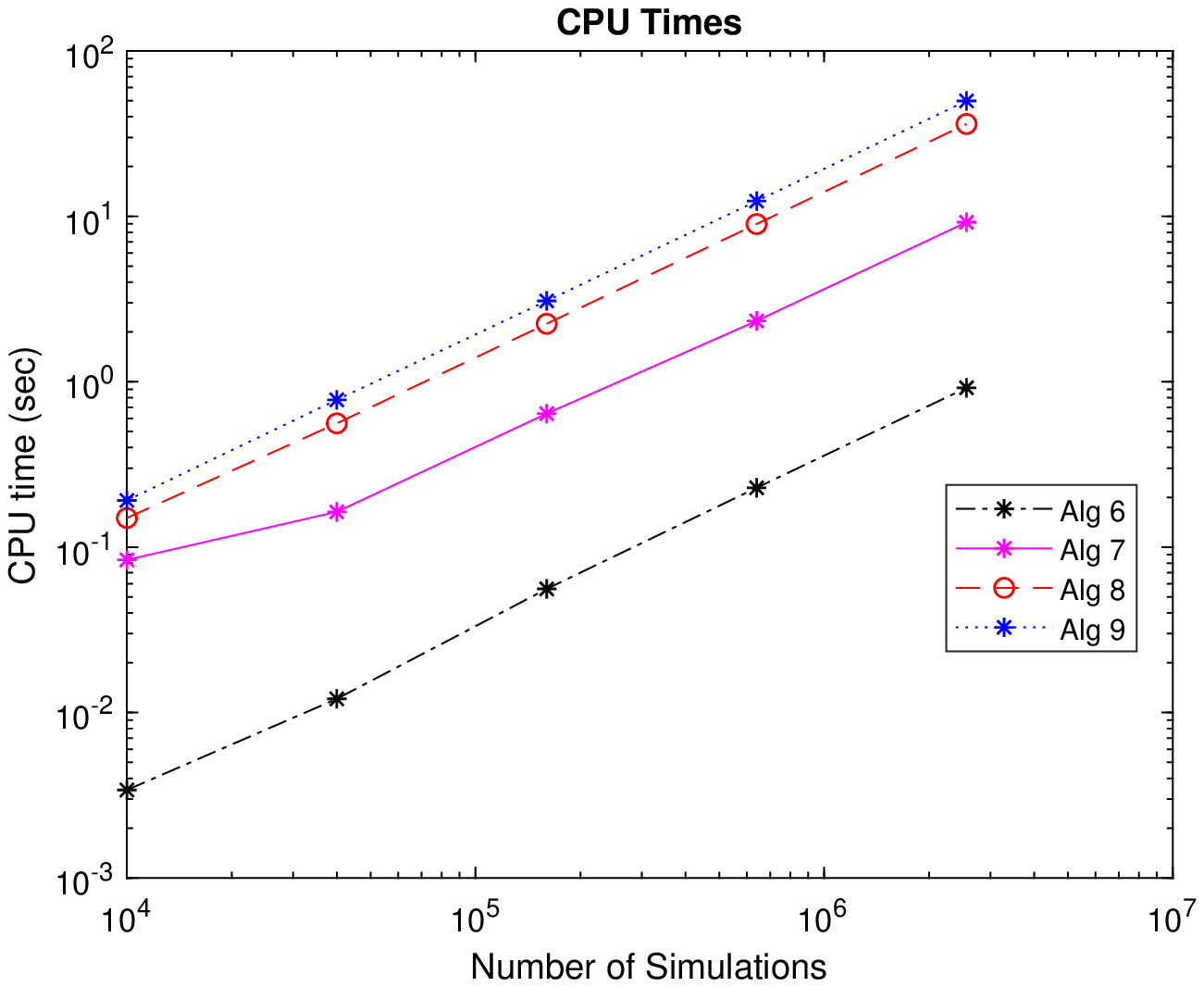}
						\caption{\footnotesize{CPU times in seconds.}}\label{fig:bgou:ComputationalTimes}
						}
					\end{subfigure}
					\begin{subfigure}[t]{.5\textwidth}{
							\centering
						\includegraphics[width=70mm]{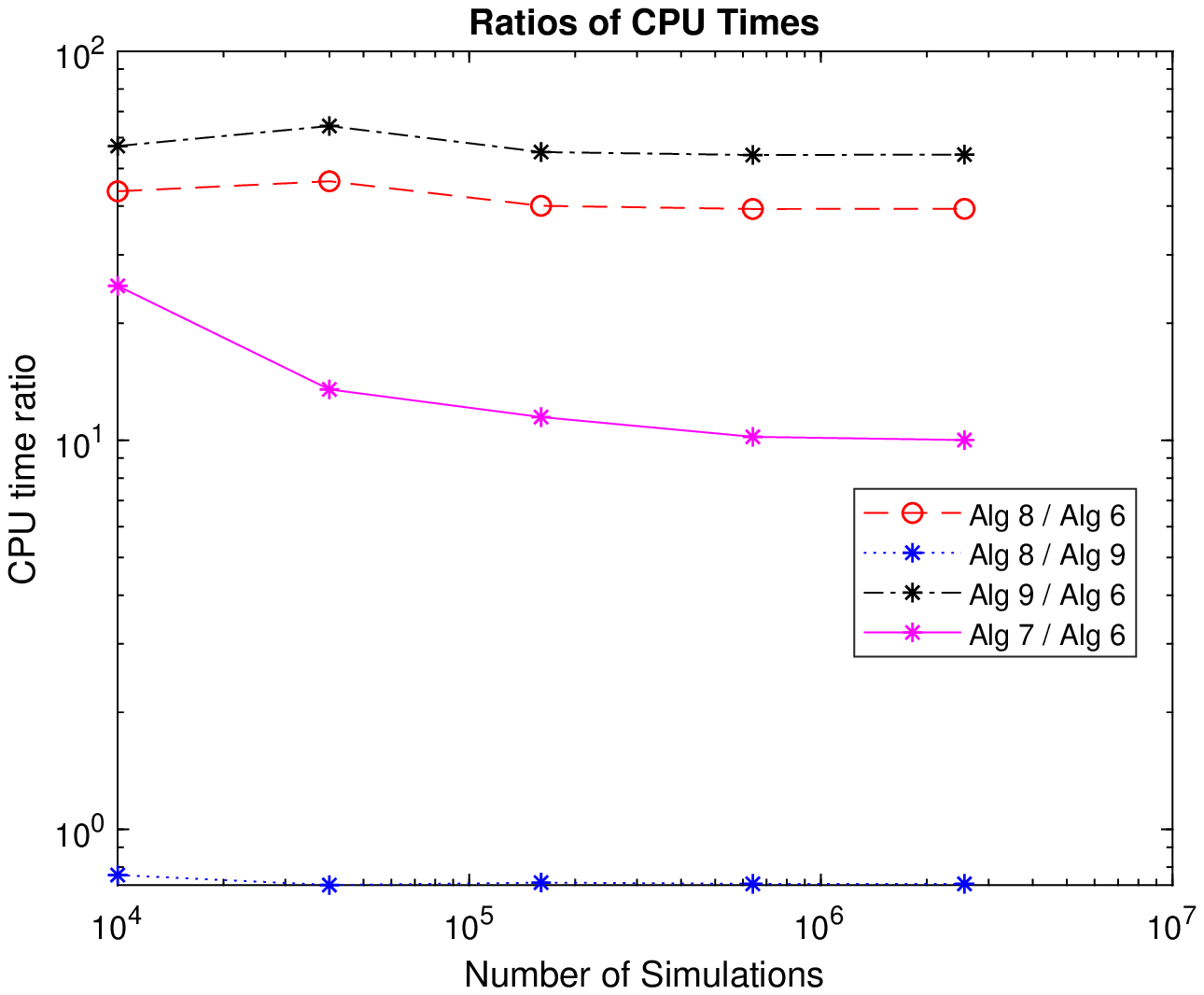}
							\caption{\footnotesize{Ratios CPU times}}\label{fig:bgou:RatioComputationalTimes}
						}
					\end{subfigure}
			\end{figure}%

        \section{ Conclusions and future inquiries\label{sec:conclusions}}

In this paper we have studied the distributional properties of the \gou\ process and its bilateral counterpart \bgou\ process. 
To this end, we have proven that in the transient regime the law
of such processes is related to the law of the \arem\ of their relative \sd\ stationary laws.
Moreover, we have shown that the \chf's and the \pdf's of such laws can be represented in closed-form as a mixture of known and tractable laws, namely, a mixture of a Polya or a Binomial distribution. 

As a simple consequence, we can design exact and efficient algorithms to generate the trajectory of a \gou\ and a \bgou\ process. 
Our numerical experiments have illustrated that our
strategy has a remarkable computational advantage and cuts the
simulation time down by a factor larger than {$30$} compared to the existing alternatives available in the literature. In particular, due to the very small computational times, they are well suitable for real-time applications. One additional advantage is that our algorithms avoid
the assumption of considering at most one jump per unit of
time.

Moreover, although not the focus of our study, knowing the density in closed-form and having simple formulas for the cumulants of the distribution, one could conceive a parameter estimation procedure based on likelihood methods, on the generalized method of moments using the analogy with the GAR(1) auto-regressive processes introduced in Gaver and Lewis \cite{gaver_lewis_1980}. 
Of course, in any practical applications, some series truncation
rule must be adopted as well as the generalization to time-dependent parameters is still open. These investigations will then be the focus of future inquires.

From the mathematical point of view,
it would also be interesting to study if --  and under which
conditions --  these results hold for other generalized
Ornstein-Uhlenbeck processes: for instance, for processes whose stationary law is a Generalized Gamma Convolutions distribution (see Bondesson \cite{Bondesson1992}) or for a Variance Gamma driven OU as discussed in Cummins et al.\cite{CKM17}.

In a primarily economic and financial perspective, the future
studies could cover the extension to  a multidimensional
setting with correlated Poisson processes as those introduced
for instance in Lindskog and McNeil\cite{LindskogMcNeil} or in
Cufaro Petroni and Sabino \cite{cs17}. A last topic deserving
further investigation is the time-reversal simulation of the \gou\ and \bgou\ processes generalizing the results of Pellegrino and
Sabino\cite{PellegrinoSabino15} and Sabino\cite{Sabino20} to the
case of the mean reverting compound Poisson processes.

	\bibliographystyle{plain}
   \bibliography{Cufaro_Sabino_Bigamma_OU}
\end{document}